\documentclass[11pt]{article}
\usepackage{amsmath,amssymb,graphicx,amsthm}
\usepackage{diagbox}
\usepackage{hyperref}
\usepackage[nosort]{cite}
\usepackage{mathtools}

\usepackage{multicol}

\usepackage{color}
\definecolor{darkred}{rgb}{0.65,0.15,0}
\definecolor{newgreen}{rgb}{0.2,0.62,0.14}
\definecolor{darkblue}{rgb}{0.1,0.15,0.7}
\definecolor{purple}{rgb}{0.35,0.1,0.55}
\hypersetup{pdfborder={0 0 0},colorlinks=true,urlcolor=blue,citecolor=blue,linkcolor=darkred,linktocpage=true}

\numberwithin{equation}{section}

\newcommand{\nn}{\nonumber}
\newcommand{\tA}{{\mathcal{A}}}
\newcommand{\tS}{{\mathcal{S}}}
\newcommand{\cX}{{\mathcal{X}}}
\newcommand{\cY}{{\mathcal{Y}}}
\newcommand{\tV}{{\mathfrak{V}}}
\newcommand{\tW}{{\mathfrak{W}}}
\newcommand{\cK}{{\mathcal{K}}}

\newcommand{\lb}{\left[}
\newcommand{\rb}{\right]}

\newcommand{\mf}[1]{{\mathfrak{#1}}}

\newcommand{\ad}{{\mathrm{ad}}}

\newcommand{\ghor}{\mathring{\mathfrak{g}}}
\newcommand{\khor}{\mathring{\mathfrak{k}}}
\newcommand{\phor}{\mathring{\mathfrak{p}}}

\newcommand{\eprint}[1]{{\href{http://arxiv.org/abs/#1}{[\texttt{#1}]}}}
\newcommand{\eprintN}[1]{{\href{http://arxiv.org/abs/#1}{[\texttt{#1 [hep-th]}]}}}

\newcommand{\eprintRT}[1]{{\href{http://arxiv.org/abs/#1}{[\texttt{#1 [math.RT]}]}}}
\newcommand{\eprintGR}[1]{{\href{http://arxiv.org/abs/#1}{[\texttt{#1 [math.GR]}]}}}
\newcommand{\doi}[2]{{\hypersetup{urlcolor=darkred}\href{http://dx.doi.org/#2}{#1}\hypersetup{urlcolor=blue}}}

\theoremstyle{remark}
\newtheorem*{rem*}{\protect\remarkname}
\theoremstyle{plain}
\newtheorem{thm}{\protect\theoremname}
\theoremstyle{remark}
\newtheorem{rem}[thm]{\protect\remarkname}
\theoremstyle{plain}
\newtheorem{lem}[thm]{\protect\lemmaname}
\theoremstyle{plain}
\newtheorem{prop}[thm]{\protect\propositionname}
\theoremstyle{plain}
\newtheorem{cor}[thm]{\protect\corollaryname}
\theoremstyle{plain}
\newtheorem*{prop*}{\protect\propositionname}
\theoremstyle{definition}
\newtheorem*{problem*}{\protect\problemname}
\theoremstyle{definition}
\newtheorem{defn}[thm]{\protect\definitionname}
\theoremstyle{plain}
\newtheorem*{conjecture*}{\protect\conjecturename}

\providecommand{\conjecturename}{Conjecture}
\providecommand{\corollaryname}{Corollary}
\providecommand{\lemmaname}{Lemma}
\providecommand{\problemname}{Problem}
\providecommand{\propositionname}{Proposition}
\providecommand{\remarkname}{Remark}
\providecommand{\theoremname}{Theorem}
\providecommand{\definitionname}{Definition}

\begin{document}

\mbox{}\\[10mm]

\begin{center}

{\LARGE \bf \sc  Representations of involutory subalgebras\\[2mm]of affine Kac--Moody algebras}\\[5mm]

\vspace{6mm}
\normalsize
{\large  Axel Kleinschmidt${}^{1,2}$, Ralf K\"ohl${}^{3,4}$, Robin Lautenbacher${}^3$ and Hermann Nicolai${}^{1}$}

\vspace{10mm}
${}^1${\it Max-Planck-Institut f\"{u}r Gravitationsphysik (Albert-Einstein-Institut)\\
Am M\"{u}hlenberg 1, DE-14476 Potsdam, Germany}
\vskip 1 em
${}^2${\it International Solvay Institutes\\
ULB-Campus Plaine CP231, BE-1050 Brussels, Belgium}
\vskip 1 em
${}^3${\it  Justus-Liebig-Universit\"at, Mathematisches Institut, Arndtstra\ss{}e 2, 35392 Gie\ss{}en, Germany}
\vskip 1 em
${}^4${\it current address: Christian-Albrechts-Universit\"at, Mathematisches Seminar, Heinrich-Hecht-Platz 6, 24118 Kiel, Germany}
\vspace{10mm}

\hrule

\vspace{5mm}

 \begin{tabular}{p{14cm}}

We consider the subalgebras of split real, non-twisted affine Kac--Moody Lie algebras that are fixed by the Cartan--Chevalley involution. These infinite-dimensional Lie algebras are not of Kac--Moody type and admit finite-dimensional unfaithful representations. We exhibit a formulation of these algebras in terms of $\mathbb{N}$-graded Lie algebras that allows the construction of a large class of representations 
using the techniques of induced representations. We study how these representations relate to previously established spinor representations as they arise in the theory of supergravity and work out a detailed example in the case of the affine extension of $\mf{e}_8$.
\end{tabular}

\vspace{6mm}
\hrule
\end{center}

\thispagestyle{empty}

\newpage
\setcounter{page}{1}

\setcounter{tocdepth}{1}
\tableofcontents

\bigskip

\section{Introduction} 

Every (split real) Kac--Moody algebra $\mf{g}=\mf{g}(A)$ for an indecomposable generalised Cartan matrix $A$ has a Cartan--Chevalley involution $\omega$ that defines an involutory (`maximal compact') Lie subalgebra $\mf{k}=\mf{k}(\mf{g})<\mf{g}$ as its subalgebra of fixed points~\cite{Kac1990}. If $\mf{g}$ is of finite type, then $\mf{k}$ is compact, whence reductive; its structure and representation theory play a key role in studies of symmetric spaces and automorphic representations. For  Kac--Moody algebras $\mf{g}$, generators and relations for $\mf{k}$ were given in the early work~\cite{Berman:1989}, but, curiously, very little is known about the representation theory of $\mf{k}$ in the infinite-dimensional case.
 One reason for this is that these algebras do not fit into 
more standard frameworks, as they do not exhibit a graded, but rather a filtered structure;
in particular, they are not Kac--Moody algebras~\cite{Kleinschmidt:2005bq}, unlike the algebras from which they 
descend. This means also that the customary tools of representation theory (lowest 
and highest weight representations, character formulas, {\em etc.}) are not applicable. This being said, we point out (see appendix~\ref{app:gim}) that the complexification of $\mf{k}$ is a quotient of a GIM algebra as defined by Slodowy~\cite{Slodowy}; therefore, the representation theory of the latter, once developed, should help with understanding the (real and complex) representations of $\mf{k}$.

A remarkable property of $\mf{k}$ is that it inherits an invariant bilinear form from $\mf{g}$. For split real $\mf{g}$ the bilinear form on $\mf{g}$ is indefinite but its restriction to $\mf{k}$ is (negative) definite and for this reason $\mf{k}$ is sometimes referred to as maximal compact. One may 
thus consider the Hilbert space completion $\widehat{\mf{k}}$ of $\mf{k}$ with respect to the norm defined by the invariant bilinear form. One of the results of our paper is that, for untwisted 
affine $\mf{g}$, the Lie bracket does not close on $\widehat{\mf{k}}$, so that $\widehat{\mf{k}}$ is 
not a Hilbert Lie algebra with respect to the standard invariant bilinear form; in fact, it is
not even a Lie algebra (see appendix~\ref{app:compl}).

The Lie algebra $\mf{k}$ and their associated groups are potentially important in physics applications, namely as infinite-dimensional generalisations of the R symmetries which govern the fermionic sectors of 
certain supersymmetric unified models of the fundamental interactions.
Perhaps the most surprising result to come out of these physics motivated studies is that 
the infinite-dimensional algebras $\mf{k}$, unlike the Kac--Moody algebras they are embedded into, 
admit non-trivial {\em finite}-dimensional, hence unfaithful, representations, in particular 
fermionic (double-valued) ones, which cannot be obtained by decomposing representations
of the ancestor Kac--Moody algebra $\mf{g}$ under its subalgebra $\mf{k}$.
These representations were originally found by
analysing the action of the infinite-dimensional duality groups arising in the dimensional 
reduction of \mbox{(super-)}gravities to space-time dimensions $D\leq 2$ on the 
fermions. For the affine case it was realised already long ago 
 \cite{Nicolai:1991tt,Nicolai:2004nv} that these actions correspond to evaluation 
representations of a novel type, involving not just the evaluation of loop group elements 
at one of two distinguished points in the complex spectral parameter plane, but also its higher 
derivatives. Likewise, for the indefinite, and more specifically hyperbolic, extensions 
of affine algebras and their involutory subalgebras, several concrete examples were 
found in terms of actions on the supersymmetry parameters
and the gravitino fields (at a given spatial point)  \cite{deBuyl:2005zy,Damour:2005zs,deBuyl:2005sch,Damour:2006xu,Kleinschmidt:2006tm,Damour:2009zc}.
These actions extend the so-called `generalised holonomy' groups in the physics 
literature~\cite{Duff:1990xz,Duff:2003ec,Hull:2003mf} to an infinite-dimensional context. Furthermore, and most relevant 
for the present work, some new finite-dimensional representations {\em beyond} supergravity 
were  identified in \cite{Kleinschmidt:2013eka,Kleinschmidt:2016ivr}, and their structure  
was further analysed in~\cite{Kleinschmidt:2018hdr}, where surprising features were discovered, such as the generic non-compactness of the 
quotient algebras (and quotient groups) obtained by dividing the original
algebra by the annihilating ideal of the given representation. These representations were further studied in~\cite{Hainke:2014,Lautenbacher:2017} where in particular the action of the corresponding (covering) group representation of $\tilde{K}(G)$~\cite{Ghatei:2015} was clarified; moreover, in complete analogy to the finite-dimensional situation it turns out that in the simple-laced case this two-fold covering group $\tilde{K}(G)$ is simply connected with respect to its natural topology \cite{Harring:2019} (induced from the Kac--Peterson topology on the corresponding Kac--Moody group).  All these studies of Lie algebra representations and corresponding covering group representations so far, while applicable to large classes of generalised Cartan matrices (simply-laced; often even any symmetrisable generalised Cartan matrix) have been limited to a small number of 
explicitly known examples of concrete representations, as all efforts to find a larger class of examples of representations and 
to understand their underlying structure and representation theory in a more systematic way have failed until now.
 
\medskip 
 
In this paper, we study the simplest case, corresponding to untwisted affine Kac--Moody algebras over $\mathbb{K}=\mathbb{R}$ or $\mathbb{K}=\mathbb{C}$. We show that there do exist infinitely many such unfaithful representations of ever increasing dimensions. Our construction finally provides a systematic {\em raison d'\^etre} 
for such representations. Here, we aim for an {\em ab initio} 
construction of $\mf{k}$ representations (in particular of spinorial nature) and not for representations that are obtained by branching representations of the affine $\mf{g}$. 
The main method, which is inspired 
by the supergravity realisation on fermions~\cite{Nicolai:2004nv}, is to map the filtered structure to a graded 
one, by replacing Laurent polynomials $\mathbb{K}[t,t^{-1}]$ in a variable $t$ by power series $\mathbb{K}[[u]]$ in another variable $u$ that is related to $t$ by~\eqref{eq:Moeb}. We shall refer to the graded structure as the \textit{parabolic model} of $\mf{k}$.
In the parabolic model based on $u$, 
representations can be constructed easily by means of the
Poincar\'e--Birkhoff--Witt (PBW) basis of the enveloping algebra of the underlying
 Lie algebra graded by powers of $u$. We will present several explicit examples (related to maximal
supergravity) to illustrate the construction, which puts in evidence the rapid growth of the
associated representations. Our results can be viewed as a prelude to the construction
of similar representations for involutory subalgebras of hyperbolic Kac--Moody algebras,
that will be required for a better understanding of the fermionic sector of unified
models, and perhaps pave the way for an embedding of the Standard Model fermions
into a unified framework~\cite{Meissner:2014joa,Kleinschmidt:2015sfa,Meissner:2018gtx}.

In a more global view, the results of this paper afford not only a completely 
new perspective on the representation theory of such algebras, but even more
importantly, open new avenues towards exploring the structure of the associated {\em groups}. 
This would be especially relevant for infinite-dimensional Kac--Moody algebras of {\em indefinite} type
for which, however, a representation theory extending the present results remains to be developed. 
While the loop group approach has proven very useful for affine Kac--Moody 
groups \cite{Pressley:1986} there is no comparable tool available for 
studying indefinite Kac--Moody groups, where often one
has to resort to `local-to-global' methods. Namely,
many mathematical observations concerning (two-spherical) Kac--Moody groups rely upon `gluing' 
the SL(2,$\mathbb{R}$) subgroups associated to the simple real roots of
the underlying Kac--Moody algebra $\mf{g}$ and on exploiting Tits' theory 
of buildings and extensions thereof \cite{Tits:1974,Kac:1985,Abramenko:1997,Carbone:2020,Marquis:2018}. However, in this way one does not gain a 
truly `global' perspective on the associated groups. The same statement applies 
to an analogous construction of the groups $K(G)$ associated to involutory subalgebras 
$\mf{k}\subset \mf{g}$ for indefinite $\mf{g}$, which would proceed by similarly `gluing' compact 
SO(2) $\subset$ SL(2,$\mathbb{R}$) subgroups associated to the simple real roots of $\mf{g}$~\cite{Ghatei:2015} (which, contrary to the Kac--Moody group $G$, in fact also works in the non-two-spherical case). By contrast, we here propose a fundamentally different approach 
which would be based on the construction of infinite sequences of larger and larger, 
but still finite-dimensional, quotient groups that, while remaining infinitely degenerate 
at each step of the sequence, capture `more and more' of the infinite-dimensional 
group (and in particular the information contained in the root spaces associated to imaginary
 roots), and in such a way that the group action is fully under control at 
each step of the construction --- in more mathematical terms, we propose to investigate whether the involutory Lie subalgebra $\mathfrak{k}$ and its corresponding group $K(G)$ might be residually finite-dimensional. In the affine case this is true for $\mf{g}$ and therefore for $\mf{k}$ as well. One of our results is that our construction recovers this for $\mf{k}$ without reference to $\mf{g}$.

Our main results are the following two theorems:
\renewcommand*{\thethm}{\Alph{thm}}
\begin{thm}
\label{thm:thm_A}
Denote by $\mathbb{P}\coloneqq \mathbb{K}\left[\left[u\right]\right]$ the ring of formal power series and let $\mathbb{P}^{\pm}$ denote the $\mathbb{K}$-submodules of formal even (resp. odd) power series in $u$. 
Let $\mf{g}\coloneqq\mathfrak{g}(A)\left(\mathbb{K}\right)$ be a split untwisted affine Kac--Moody algebra over $\mathbb{K}$ with 'maximal compact' subalgebra $\mf{k}$ fixed by the Cartan--Chevalley involution, let $\ghor=\khor\oplus\phor$ be the Cartan decomposition of the distinguished classical subalgebra  $\ghor$ of $\mf{g}$  and set
\begin{equation*}
\mathfrak{N}\left(\mathbb{P}\right)\coloneqq \Big(\mathbb{P}^{+}\otimes\mathring{\mathfrak{k}}\Big)\oplus\Big(\mathbb{P}^{-}\otimes\mathring{\mathfrak{p}}\Big),\quad \left[P\otimes x,Q\otimes y\right]=\left(P\cdot Q\right)\otimes\left[x,y\right],
\end{equation*}
where the bracket on the right-hand side denotes the $\mathring{\mathfrak{g}}$-bracket. Then there exist nonsurjective monomorphisms $\rho_{\pm}:\mathfrak{k}\rightarrow\mathfrak{N}\left(\mathbb{P}\right)$, whose precise formula is given in ~\eqref{eq:homans}.

Set $\mathbb{P}_{N}\coloneqq\mathbb{P}\diagup\mathcal{I}_N$, where $\mathcal{I}_N\coloneqq\left(u^{N+1}\right)$ is the ideal generated by the element $u^{N+1}$. Projection from $\mathbb{P}$ to $\mathbb{P}_N$ induces homomorphisms of Lie algebras $\pi_N: \mathfrak{N}\left(\mathbb{P}\right)\rightarrow \mathfrak{N}\left(\mathbb{P}_N\right)$ such that  the $\rho_{\pm}^{(N)}:=\rho_{\pm}\circ\pi_N:\mathfrak{k}\rightarrow\mathfrak{N}\left(\mathbb{P}_N\right)$ are epimorphisms.
\end{thm}

\begin{thm}
\label{thm:thm_B}
Let $V$ be a faithful $\khor$-module and denote by $\tV$ the induced $\mathfrak{N}\left(\mathbb{K}\left[u\right]\right)$-module. It inherits an $\mathbb{N}$-grading $\tV=\bigoplus_{i=0}^{\infty} \tV_i$ from $\mathfrak{N}\left(\mathbb{K}\left[u\right]\right)$ such that $\tV^{(N)} = \bigoplus_{i=N}^\infty \tV_i$ is an invariant submodule. 
The quotient $\tV / \tV^{(N)}$  is a finite-dimensional $\mf{N}(\mathbb{P}_N)$-module  and therefore provides a representation of $\mf{k}$. 

On $\overline{\mf{V}}\coloneqq\left\{ \left(v_{i}\right)\ \vert\ v_{i}\in \mf{V}_{i}\right\}$, the formal completion of $\tV$, there exists a natural, faithful $\mf{N}\left(\mathbb{K}[[u]]\right)$-action. Furthermore, $\overline{\mf{V}}$ is the inverse limit of the quotients $\tV / \tV^{(N)}$ which  shows that $\mf{k}$ is residually finite-dimensional.
\end{thm}

\renewcommand*{\thethm}{\arabic{thm}}
\setcounter{thm}{0}
The paper is organised as follows: In section \ref{sec:aff} we fix our notation for the involutory subalgebra $\mf{k}$ of an untwisted affine Kac--Moody algebra. We construct the parabolic analogues $\mf{N}\left(\mathbb{K}[[u]]\right)$ and $\mf{N}\left(\mathbb{P}_N\right)$ of $\mf{k}$ over different rings in section \ref{sec:par} and show that there exist homomorphisms from $\mf{k}$ to them. Furthermore we establish that the homomorphism from $\mf{k}$ to $\mf{N}\left(\mathbb{K}[[u]]\right)$ is injective and thus proving theorem \ref{thm:thm_A}. In section \ref{sec:verm}  we take a slight detour in first constructing induced representations of $\mf{N}\left(\mathbb{K}[u]\right)$ from representations of its classical subalgebra $\khor$. Although there is no homomorphism from $\mf{k}$ to $\mf{N}\left(\mathbb{K}[u]\right)$ we will show that these induced representations can be \emph{truncated} to provide finite-dimensional representations of $\mf{N}\left(\mathbb{P}_N\right)$ and $\mf{k}$. We will conclude section \ref{sec:verm} by proving that the inverse limit of these representations provides a fatihful infinite-dimensional representation of $\mf{k}$, thus proving theorem~\ref{thm:thm_B}. We specialise our results to the case $\mf{k}\left(\mf{e}_9\right)$ in section \ref{sec:Specialisation} where we also describe the problem of analysing the finite-dimensional representations' structure in more detail.

\vspace{1cm} 
\noindent
 {\bf Acknowledgments:} 
 We thank Benedikt K\"onig for comments on the manuscript.
 The work of  H.N. has received funding from the European Research 
 Council (ERC) under the  European Union's Horizon 2020 research and 
 innovation programme (grant agreement No 740209). The work of R.K. and R.L. has received funding from the Deutsche Forschungsgemeinschaft (DFG) via the grant KO 4323/13-2. The work of R.L. has received funding from the Studienstiftung des deutschen Volkes. Furthermore, R.K. and R.L. gratefully acknowledge the hospitality of the Max-Planck-Gesellschaft during several extended visits at the Max-Planck-Institute for Gravitational Physics at Golm, Potsdam.

\section{Affine Kac--Moody algebras and Cartan--Chevalley involution}
\label{sec:aff}

Let $A$ be an indecomposable generalised Cartan matrix of untwisted affine type and let $\mathbb{K}$ denote $\mathbb{R}$ or $\mathbb{C}$, see~\cite[\S4]{Kac1990} for a complete list of the associated diagrams $\mathcal{D}(A)$.
The Kac--Moody algebra $\mathfrak{g}\coloneqq \mathfrak{g}(A)\left(\mathbb{K}\right)$
can be constructed explicitly as an extension of the loop algebra
$\mathfrak{L}\left(\mathring{\mathfrak{g}}\right)$, where $\mathring{\mathfrak{g}}$
denotes the classical Lie-subalgebra $\mathring{\mathfrak{g}}<\mathfrak{g}(A)\left(\mathbb{K}\right)$
that one obtains by deleting the affine node in the generalised Dynkin
diagram $\mathcal{D}(A)$. Denote by $\omega$ the Cartan--Chevalley
involution on $\mathfrak{g}(A)\left(\mathbb{K}\right)$~\cite[\S1]{Kac1990} 
and set 
\begin{equation}
\mathfrak{k}\coloneqq \mathfrak{k}(A)\left(\mathbb{K}\right)=\mathrm{Fix}_{\omega}\left(\mathfrak{g}(A)\left(\mathbb{K}\right)\right).
\end{equation}
$\mathfrak{k}(A)\left(\mathbb{K}\right)$ is called the
maximal compact subalgebra of $\mathfrak{g}(A)\left(\mathbb{K}\right)$ since the restriction of the standard invariant bilinear form on $\mf{g}$ to $\mf{k}$ is (negative) definite.
Let us start with a description of $\mathfrak{k}$ that is adapted
to the presentation of $\mathfrak{g}$ in terms of the loop algebra
$\mathfrak{L}\left(\mathring{\mathfrak{g}}\right)$. In section
\ref{sec:Specialisation}, we will complement this by a collection
of correspondences to the basis described in \cite[\S4]{Kleinschmidt:2006dy} for the particular case $\mf{g}=\mf{e}_9$, the affine extension of $\mf{e}_8$.

Denote by $\mathring{\Delta}$ the root system of $\mathring{\mathfrak{g}}$
and its Cartan--Chevalley involution\footnote{Note that the abstract Cartan--Chevalley involution of $\mathring{\mathfrak{g}}=\mathfrak{g}(\mathring{A})$
as a Kac--Moody algebra of finite type agrees with the restriction
of $\omega:\mathfrak{g}\rightarrow\mathfrak{g}$ to $\mathring{\mathfrak{g}}<\mathfrak{g}$
because the Dynkin diagram of $\mathring{\mathfrak{g}}$ is
a sub-diagram of $\mathcal{D}\left(A\right)$ due to our assumption that $A$ is untwisted.} by $\mathring{\omega}$. The $\pm1$ eigenspaces of $\mathring{\omega}$
provide the Cartan decomposition $\mathring{\mathfrak{g}}=\mathring{\mathfrak{k}}\oplus\mathring{\mathfrak{p}}$.
In terms of a Cartan basis $\big\{ E_{\gamma}\ \vert\ \gamma\in\mathring{\Delta}\big\} \cup\big\{ h_{1},\dots,h_{d}\in\mathring{\mathfrak{h}}\big\} $,
this decomposition is realised as 
\begin{equation}
\mathring{\mathfrak{k}}=\text{span}\left\{ E_{\gamma}-E_{-\gamma}\ \vert\ \gamma\in\mathring{\Delta}\right\} ,\ \quad\mathring{\mathfrak{p}}=\text{span}\left(\left\{ E_{\gamma}+E_{-\gamma}\ \vert\ \gamma\in\mathring{\Delta}\right\} \cup\mathring{\mathfrak{h}}\right),
\end{equation}
where $\mathring{\mathfrak{h}}$ denotes the Cartan subalgebra
of $\mathring{\mathfrak{g}}$. Recall that $\mathring{\mathfrak{p}}$
is a $\mathring{\mathfrak{k}}$-module but not a Lie-algebra
since $\left[\mathring{\mathfrak{p}},\mathring{\mathfrak{p}}\right]=\mathring{\mathfrak{k}}$.
 Denote by $\mathfrak{L}$ the ring of Laurent polynomials over $\mathbb{K}$. Then
the loop algebra of $\mathfrak{L}\left(\mathring{\mathfrak{g}}\right)$
is given by the tensor product $\mathfrak{L}\otimes\mathring{\mathfrak{g}}$
with the commutator given by the bilinear extension of
\begin{equation}
\left[P\otimes x,Q\otimes y\right]=\left(PQ\right)\otimes\left[x,y\right],
\end{equation}
where $\left[\cdot,\cdot\right]$ on the right-hand side denotes the bracket on $\mathring{\mathfrak{g}}$. 
As $t^{n}$, $t^{-n}$ (for $n\in\mathbb{N}=\{0,1,\ldots\}$) span $\mathfrak{L}$, the Lie algebra $\mathfrak{L}\left(\mathring{\mathfrak{g}}\right)$ is spanned by
 $\left\{ t^{\pm n}\otimes E_{\gamma}\ \vert \ \gamma\in \mathring{\Delta}\right\} \cup \left\{ t^{\pm n}\otimes h\ \vert \ h\in\mathring{\mf{h}}\right\}$ for $n\in \mathbb{N}$. 

 Now (see
\cite[\S7]{Kac1990} or \cite{Goddard:1986bp})
\begin{align}
\label{eq:aff}
\mathfrak{g}=\mathfrak{L}\left(\mathring{\mathfrak{g}}\right)\oplus\mathbb{K}\cdot K\oplus\mathbb{K}\cdot d
\end{align}
with $K,d\in\mathfrak{h}$ so that $\mathfrak{k}\subset\mathfrak{L}\left(\mathring{\mathfrak{g}}\right)$
because $\omega\left(h\right)=-h$ $\forall\,h\in\mathfrak{h}$. Thus,
in order to describe $\mf{k}= \mathfrak{k}(A)\left(\mathbb{K}\right)$
we only need to study the loop algebra $\mathfrak{L}\left(\mathring{\mathfrak{g}}\right)$
and in particular do not need to consider the central extension of
$\mathfrak{L}\left(\mathring{\mathfrak{g}}\right)$ by a 2-cocycle, giving rise to $K$. The Cartan--Chevalley involution restricted to $\mathfrak{L}\left(\mathring{\mathfrak{g}}\right)$
is given by the linear extension of 
\begin{equation}
\omega:\,P\left(t\right)\otimes x\mapsto P\left(t^{-1}\right)\otimes\mathring{\omega}\left(x\right)\ \forall\,P\in\mathfrak{L}.
\end{equation}
Denote by $\mathfrak{L}^{\pm}$ the $\pm1$ eigenspaces of $\mathfrak{L}$
under the involution $\eta:\,t\mapsto t^{-1}$. Observe that evaluation of Laurent polynomials at $t=\pm1$ is the only evaluation map that commutes with $\eta$ which is why we refer to $t=\pm1$ as the fixed points of $\eta$. The fixed point set of $\omega$ can be constructed from $\mathfrak{L}^{\pm}$, $\mathring{\mathfrak{k}}$ and $\mathring{\mathfrak{p}}$:
\begin{align*}
\omega\left(P\otimes x\right)=P\otimes x\quad \forall\,P\in\mathfrak{L}^{+},x\in\mathring{\mathfrak{k}},\quad\quad\omega\left(Q\otimes y\right)=Q\otimes y\quad \forall\,Q\in\mathfrak{L}^{-},y\in\mathring{\mathfrak{p}}
\end{align*}
From this we arrive at the following explicit realisation of $\mf{k}$:
\begin{equation}
\label{eq:kdef}
\Rightarrow\ \mathfrak{k}\left(A\right)\left(\mathbb{K}\right)=\left(\mathfrak{L}^{+}\otimes\mathring{\mathfrak{k}}\right)\oplus\left(\mathfrak{L}^{-}\otimes\mathring{\mathfrak{p}}\right).
\end{equation} 

\begin{rem}
\begin{enumerate}
\item[(i)]
The Laurent polynomials in $\mf{L}^\pm$ are spanned by $t^n\pm t^{-n}$ for $n\in \mathbb{N}$.  The product of two such basis Laurent polynomials is for example
\begin{align}
\label{eq:filt}
(t^m + t^{-m}) (t^n - t^{-n}) = \left( t^{m+n}- t^{-(m+n)} \right) - \mathrm{sgn}(m-n) \left( t^{|m-n|} - t^{-|m-n|}\right)\,.
\end{align}
We shall refer to this as a \textit{filtered} structure on $\mf{L}$ that by~\eqref{eq:kdef} induces a filtered structure on $\mf{k}$.
\item[(ii)]
The construction in (\ref{eq:kdef}) can be generalised to the case where $\mf{L}$ denotes a finitely generated commutative $\mathbb{K}$-algebra with involution and $\mf{L}^{\pm}$ denote the $\pm1$ eigenspaces with respect to this involution. This produces an analogue of maximal compact subalgebras to the case where $\mf{g}$ is a generalised current algebra. However, our interest lies in studying different models of $\mf{k}$. In particular, we will explore the relation of (\ref{eq:kdef}) to the cases $\mf{L}=\mathbb{K}[[u]]$, $\mathbb{K}[u]$ and $\mathbb{K}[u]\diagup \mathcal{I}_N$, where $\mathcal{I}_N$ is the ideal generated by the monomial $u^{N+1}$.
\item[(iii)] 
It is well known that one can form the semi-direct sum of any affine Lie algebra with the Virasoro algebra of centrally extended infinitesimal diffeomorphisms of the circle~\cite{Kac1990,Goddard:1986bp}, where both central extensions are identified. This structure descends to $\mf{k}$ and a `maximal compact' subalgebra of the Virasoro algebra~\cite{Julia}. As this will play no role in our general analysis, we defer its discussion to section~\ref{sec:Specialisation}.
\item[(iv)] 
The Lie algebra $\mf{k}$ comes with a definite and invariant bilinear form that is inherited from $\mf{g}$. In appendix~\ref{app:compl}, we show that the Hilbert space completion with respect to this norm is not compatible with the Lie algebra structure to form a Hilbert Lie algebra.
\end{enumerate}
\end{rem}

\section{Reformulation in terms of parabolic algebras}
\label{sec:par}

In this section, we construct a Lie algebra monomorphism from $\mathfrak{k}$ to a larger Lie algebra by replacing Laurent polynomials with formal power series. This corresponds to an expansion around the fixed points $t=\pm 1$ of $\eta$ instead of $0$ as described in \cite{Nicolai:2004nv, Kleinschmidt:2006dy}.

Denote by $\mathbb{P}\coloneqq \mathbb{K}\left[\left[u\right]\right]$ the
ring of formal power series together with the involutive ring automorphism
$\eta:u^{k}\mapsto(-1)^{k}u^{k}$. We refer to a formal power series
in $\mathbb{P}$ either as $\sum_{k=0}^{\infty}a_{k}u^{k}$ or just
as $\left(a_{k}\right)_{k\geq0}$ with $a_{k}\in\mathbb{K}$. Denote
the $\pm1$ eigenspaces of $\mathbb{P}$ with respect to $\eta$ by
$\mathbb{P}^{\pm}$, i.e., the formally even/odd power series in $u$. One has
\begin{equation*}
\mathbb{P}^{+}\coloneqq \left\{ \left(a_{k}\right)_{k\in\mathbb{N}_0}\ \vert\ a_{2k}\in\mathbb{K},\,a_{2k+1}=0\ \forall\,k\in\mathbb{N}\right\} ,
\end{equation*}
\begin{equation*}
\mathbb{P}^{-}\coloneqq\left\{ \left(a_{k}\right)_{k\in\mathbb{N}_0}\ \vert\ a_{2k+1}\in\mathbb{K},\,a_{2k}=0\ \forall\,k\in\mathbb{N}\right\} .
\end{equation*}
We mimic the loop algebra construction by setting
\begin{equation}
\label{eq:Ndef}
\mathfrak{N}\left(\mathbb{K}[[u]]\right)\coloneqq \Big(\mathbb{P}^{+}\otimes\mathring{\mathfrak{k}}\Big)\oplus\Big(\mathbb{P}^{-}\otimes\mathring{\mathfrak{p}}\Big),\quad \left[P\otimes x,Q\otimes y\right]=\left(P\cdot Q\right)\otimes\left[x,y\right],
\end{equation}
where the bracket on the right-hand side again denotes the $\mathring{\mathfrak{g}}$-bracket.
Recall that multiplication in the ring of formal power series is given
by convolution, i.e., for $P=\sum_{n=0}^{\infty}a_{n}u^{n}$ and $Q=\sum_{n=0}^{\infty}b_{n}u^{n}$
one has 
\begin{align}
\label{conv}
P\cdot Q=\sum_{n=0}^{\infty}c_{n}u^{n},\quad c_{n}=\sum_{k=0}^{n}a_{k}b_{n-k}.
\end{align}

\begin{rem}\label{rmk:ser}
One could also consider replacing $\mathbb{P}$ by
the ring of polynomials $\mathbb{K}[u]$. In contrast to $\mathfrak{L}^{\pm}$
which have a filtered structure, $\mathbb{K}[u]$ has a gradation given
by the degree of polynomials. This extends to a gradation on $\big(\mathbb{K}[u]^{+}\otimes\mathring{\mathfrak{k}}\big)\oplus\left(\mathbb{K}[u]^{-}\otimes\mathring{\mathfrak{p}}\right)$.
If we were to construct a monomorphism from $\mf{k}$ using only $\mathbb{K}[u]$
one would expect to be able to pull this gradation back to $\mathfrak{k}\left(A\right)\left(\mathbb{K}\right)$.
Since we doubt this to be possible we work over $\mathbb{K}\left[\left[u\right]\right]$ which arises
as the formal completion of $\mathbb{K}[u]$. There we do not have
a gradation by degree of polynomials any more in the proper sense
because for this any element in $\mathbb{K}\left[\left[u\right]\right]$
would have to decompose into the sum of finitely many homogeneous
elements. This is true for polynomials but not for formal power series.
Thus, we sacrifice a graded structure in order to achieve injectivity. However, in section \ref{sec:verm} we will consider this case as a preliminary step.
\end{rem}

As the power series $\mathbb{K}[[u]]$ are associated with the expansion around the fixed point $t=\pm 1$ rather than $t=0$, we construct a homomorphism by relating the expansion via a (M\"obius-type) transformation 
\begin{align}
\label{eq:Moeb}
u= \frac{1\mp t}{1\pm t} 
\quad\Leftrightarrow\quad
t= \pm \frac{1-u}{1+u}
\end{align}
through the Taylor series ($n\geq 0$)
\begin{align}
\label{eq:Tayl}
t^n + t^{-n} = (\pm 1)^n \sum_{k\geq 0} a_{2k}^{(n)} u^{2k}\,,
\hspace{10mm}
t^n - t^{-n} = (\pm 1)^n\sum_{k\geq 0} a_{2k+1}^{(n)} u^{2k+1}\,.
\end{align}

 The filtered multiplication of the Laurent polynomials in $\mf{L}^{\pm}$ is then captured by the following lemma.

\begin{lem}\label{lem1}
For each $n\in\mathbb{N}$ the sequences $\left(a_{2k}^{(n)}\right)_{k\in\mathbb{N}}$
and  $\left(a_{2k+1}^{(n)}\right)_{k\in\mathbb{N}}$ given by
given by
\begin{equation}
a_{2k}^{(n)}=2 \sum_{\ell=0}^{n}\begin{pmatrix}2n\\
2\ell
\end{pmatrix}\begin{pmatrix}k-\ell+n-1\\
k-\ell
\end{pmatrix},\quad a_{2k+1}^{(n)}=-2 \sum_{\ell=0}^{ n-1 }\begin{pmatrix}2n\\
2\ell+1
\end{pmatrix}\begin{pmatrix}k-\ell+n-1\\
k-\ell
\end{pmatrix}\label{eq:coeffs}
\end{equation}
satisfy
\begin{subequations}
\label{eq:convo}
\begin{align}
\sum_{\ell=0}^{k}a_{2\ell}^{(m)}a_{2(k-\ell)}^{(n)} & =  a_{2k}^{(m+n)}+a_{2k}^{(\vert m-n\vert)}\label{eq:convo1}\\
\sum_{\ell=0}^{k}a_{2\ell}^{(m)}a_{2(k-\ell)+1}^{(n)} & =  a_{2k+1}^{(m+n)}+\mathrm{sgn}(n-m)a_{2k+1}^{(\vert m-n\vert)}\label{eq:convo2}\\
\sum_{\ell=0}^{k-1}a_{2\ell+1}^{(m)}a_{2(k-\ell)-1}^{(n)} & =  a_{2k}^{(m+n)}-a_{2k}^{(\vert m-n\vert)}.\label{eq:convo3}
\end{align}
\end{subequations}
The coefficients $(-1)^{n}a_{2k}^{(n)}$ and $(-1)^{n}a_{2k+1}^{(n)}$
also satisfy eqs. (\ref{eq:convo1}), (\ref{eq:convo2}) and (\ref{eq:convo3}). Furthermore, for fixed $n\in\mathbb{N}^*=\{1,2,\ldots\}$, the values of $a_{2k}^{(n)}$ and $a_{2k+1}^{(n)}$ are given by polynomials  in $k$ of degree $n-1$ for all $k\in\mathbb{N}^*$.
\end{lem}

\begin{rem}
 The  transformations~\eqref{eq:Moeb} have the property that they interchange 
 the points $0$ and $\infty$ (that are exchanged by the involution $t\leftrightarrow t^{-1}$) with the fixed points $\pm1$ of $\eta$. The maps $t\mapsto u(t)$ in~\eqref{eq:Moeb} are fixed uniquely by the requirement that $(0,\infty,+1,-1)\mapsto (+1,-1,0,\infty)$ and $(0,\infty,+1,-1)\mapsto (+1,-1,\infty,0)$, respectively.
\end{rem}

\begin{proof}
We first compute the Taylor series for the M\"obius transformation~\eqref{eq:Moeb} for $n\in \mathbb{N}$
\begin{align*}
t^n + t^{-n} &=  (\pm 1)^{n}\frac{2}{(1-u^2)^n}  \sum_{\ell=0}^{n} \binom{2n}{2\ell} u^{2\ell} = 2(\pm 1)^n \sum_{k\geq 0} \sum_{\ell=0}^n \binom{2n}{2\ell}\binom{n+k-1}{k} u^{2(k+\ell)}\,,
\end{align*}
from which the first formula in~\eqref{eq:coeffs} follows. The second identity in~\eqref{eq:coeffs} is deduced similarly from the expansion of $t^n - t^{-n}$.  The order of the polynomial follows from the binomial coefficient $\tbinom{n+k-\ell-1}{k-\ell}$.

The convolution properties~\eqref{eq:convo} follow from multiplying out
\begin{align*}
\left( t^m  + t^{-m} \right) \left( t^n + t^{-n}\right)  = t^{m+n} + t^{-(m+n)} + t^{|m-n|} + t^{-|m-n|}
\end{align*}
for~\eqref{eq:convo1} for the first one upon using~\eqref{conv}. The identities~\eqref{eq:convo2} and~\eqref{eq:convo3} follow in the same way by changing the appropriate signs.
\end{proof}

We collect some formulae concerning the coefficients $a_{2k}^{(n)}$ and $a_{2k+1}^{(n)}$. One has $a_0^{(n)}=2$ for all $n\in\mathbb{N}$. For $k>0$ the first few coefficients $a_{2k}^{(n)}$ are given by the following polynomials
in $k$ (note that these expressions are not valid for $k=0$): 
\begin{align}
a_{2k}^{(0)}=0,\quad a_{2k}^{(1)}=4,\quad a_{2k}^{(2)}=16k,\quad a_{2k}^{(3)}=32k^{2}+4,\quad a_{2k}^{(4)}=\frac{128}{3} k^{3}+\frac{64}{3} k,
\end{align}
while the first few coefficients $a_{2k+1}^{(n)}$ are given
by polynomials in $k$ as well:
\begin{align}
a_{2k+1}^{(1)}&=-4,\ a_{2k+1}^{(2)}=-16 k-8,\ a_{2k+1}^{(3)}=-32 k^{2}-32 k-12,\nn\\
 a_{2k+1}^{(4)}&=-\frac{128}{3} k^{3}-64 k^{2}-\frac{160}{3} k-16.
\end{align}

\medskip

With lemma~\ref{lem1} above one can now construct a Lie algebra homomorphism from $\mf{k}$ to $\mf{N}(\mathbb{K}[[u]])$, defined in~\eqref{eq:Ndef}, that is constructed using~\eqref{eq:Tayl}.
\begin{prop}
\label{prop:ansatz for hom}The linear map $\rho_{\pm}:\mathfrak{k}(A)\left(\mathbb{K}\right)\rightarrow\mathfrak{N}\left(\mathbb{K}[[u]]\right)$
defined by 
\begin{subequations}
\label{eq:homans}
\begin{align}
\left(t^{n}+t^{-n}\right)\otimes x  &\mapsto\left(\pm1\right)^{n}\sum_{k=0}^{\infty}a_{2k}^{(n)}u^{2k}\otimes x\quad\forall\,x\in\mathring{\mathfrak{k}}\label{eq:ansatz for k}\\
\left(t^{n}-t^{-n}\right)\otimes y  &\mapsto\left(\pm1\right)^{n}\sum_{k=0}^{\infty}a_{2k+1}^{(n)}u^{2k+1}\otimes y\quad\forall\,y\in\mathring{\mathfrak{p}}\label{eq:ansatz for p}
\end{align}
\end{subequations}
with $a_{2k}^{(n)}$, $a_{2k+1}^{(n)}$ as in (\ref{eq:coeffs}) extends
to a homomorphism of Lie algebras.
\end{prop}

\begin{proof}
This can be checked case by case for pairs $(a,b)\in\big(\mathfrak{L}^{+}\otimes\mathring{\mathfrak{k}}\big)\times\big(\mathfrak{L}^{+}\otimes\mathring{\mathfrak{k}}\big),\,$$\big(\mathfrak{L}^{+}\otimes\mathring{\mathfrak{k}}\big)\times\left(\mathfrak{L}^{-}\otimes\mathring{\mathfrak{p}}\right)$,
$\left(\mathfrak{L}^{-}\otimes\mathring{\mathfrak{p}}\right)\times\left(\mathfrak{L}^{-}\otimes\mathring{\mathfrak{p}}\right)$
and extending the result by (bi-)linearity. We demonstrate this for the first pair, the other pairs work analogously.

First, set $x^{(n)}\coloneqq\left(t^{n}+t^{-n}\right)\otimes x$
for all $x\in\mathring{\mathfrak{k}}$.
Then by definition
\begin{equation*}
\left[x_{1}^{(m)},x_{2}^{(n)}\right] =  \left(t^{m+n}+t^{-(m+n)}\right)\otimes\left[x_{1},x_{2}\right]+\left(t^{|m-n|}+t^{-|m-n|}\right)\otimes\left[x_{1},x_{2}\right],
\end{equation*}
so that from~\eqref{eq:ansatz for k}
\[
\rho_{\pm}\left(\left[x_{1}^{(m)},x_{2}^{(n)}\right]\right)=\left(\pm1\right)^{m+n}\sum_{k=0}^{\infty}\left(a_{2k}^{(m+n)}+a_{2k}^{(\vert m-n\vert)}\right)u^{2k}\otimes\left[x_{1},x_{2}\right].
\]
On the other hand 
\begin{align*}
\left[\rho\left(x_{1}^{(m)}\right),\rho\left(x_{2}^{(n)}\right)\right] & =  \left(\pm1\right)^{m+n}\sum_{k=0}^{\infty}\sum_{\ell=0}^{k}a_{2(k-\ell)}^{(m)}a_{2\ell}^{(n)}u^{2k}\otimes\left[x_{1},x_{2}\right]\\
 & =  \left(\pm1\right)^{m+n}\sum_{k=0}^{\infty}\left(a_{2k}^{(m+n)}+a_{2k}^{(\vert m-n\vert)}\right)u^{2k}\otimes\left[x_{1},x_{2}\right],
\end{align*}
by (\ref{eq:convo1}), proving the claim. For the other pairs one proceeds similarly, using in particular $a_{0}^{(n)}=2$ for all $n$, and 
this shows that the map $\rho_{\pm}$ extends
to a homomorphism of Lie algebras.
\end{proof}

It is now possible to introduce a cutoff in  (\ref{eq:homans}) in order to make all expressions finite sums.
Denote by 
\begin{align}
\label{eq:idN}
\mathcal{I}_N\coloneqq\left(u^{N+1}\right)
\end{align}
the ideal in $\mathbb{K}[[u]]$ that is generated by the element $u^{N+1}$ and set $\mathbb{P}_{N}\coloneqq\mathbb{K}[[u]]\diagup\mathcal{I}_N$. Furthermore, define the even and odd parts as
\[
\mathbb{P}_{N}^{+}\coloneqq\text{span}_{\mathbb{K}}\left\{ u^{2k}+\mathcal{I}_N\vert k=0,\dots,\left\lfloor \frac{N}{2}\right\rfloor\right\} ,\ \mathbb{P}_{N}^{-}\coloneqq\text{span}_{\mathbb{K}}\left\{ u^{2k+1}+\mathcal{I}_N\vert k=\left\lfloor \frac{N-1}{2}\right\rfloor\right\} 
\]
and consider the Lie algebra 
\begin{align}
\label{eq:NNdef}
\mathfrak{N}\left(\mathbb{P}_N\right)\coloneqq\left(\mathbb{P}_{N}^{+}\otimes\mathring{\mathfrak{k}}\right)\oplus\left(\mathbb{P}_{N}^{-}\otimes\mathring{\mathfrak{p}}\right), \quad \left[P\otimes x,Q\otimes y\right]=\left(P\cdot Q\right)\otimes\left[x,y\right]
\end{align}
where the bracket on the right-hand side still denotes the $\mathring{\mathfrak{g}}$-bracket.

\begin{rem}
Note that the ring of formal power series $\mathbb{P}$ is isomorphic to the inverse limit of the inverse system $\left(\mathbb{P}_N\right)_{N\in\mathbb{N}}$ with the obvious bonding maps~\cite{Rotman}. 
As a consequence $\mathfrak{N}\left(\mathbb{K}[[u]]\right)$ is the inverse limit of the $\mathfrak{N}\left(\mathbb{P}_N\right)$:
\begin{align}
\mf{N}(\mathbb{K}[[u]]) = \lim_{\substack{\longleftarrow \\ N\to \infty}} \mf{N}(\mathbb{P}_N).
\end{align}
In the following we will construct homomorphisms $\rho_{\pm}^{(N)}:\mathfrak{k}(A)\left(\mathbb{K}\right)\rightarrow\mathfrak{N}\left(\mathbb{P}_N\right)$ which are then shown to be surjective. As $\rho_{\pm}^{(N)}$ is constructed via $\rho_{\pm}$ and the natural projection $\pi_N$ from $\mathfrak{N}\left(\mathbb{K}[[u]]\right)$ to $\mathfrak{N}\left(\mathbb{P}_N\right)$ the $\rho_{\pm}^{(N)}$ are compatible with the natural projections  from $\mathfrak{N}\left(\mathbb{P}_{N+M}\right)$ to $\mathfrak{N}\left(\mathbb{P}_N\right)$. Assuming we had started with the $\rho_\pm^{(N)}$ then by the universal property of inverse limits we would have obtained $\rho_\pm$ as the unique homomorphism $\rho_\pm: \mathfrak{k}(A)\left(\mathbb{K}\right)\rightarrow\mathfrak{N}\left(\mathbb{K}[[u]]\right)$ such that $\rho_\pm^{(N)}=\rho_\pm\circ \pi_N$. 
\end{rem}

\begin{cor}
\label{cor:truncated hom}The homomorphisms of Lie algebras $\rho_{\pm}:\mathfrak{k}(A)\left(\mathbb{K}\right)\rightarrow\mathfrak{N}\left(\mathbb{K}[[u]]\right)$
induce homomorphisms $\rho_{\pm}^{(N)}:\mathfrak{k}(A)\left(\mathbb{K}\right)\rightarrow\mathfrak{N}\left(\mathbb{P}_N\right)$
which are given explicitly by
\begin{align}
\left(t^{n}+t^{-n}\right)\otimes x  &\mapsto\left(\pm1\right)^{n}\sum_{k=0}^{\left\lfloor N/2\right\rfloor }a_{2k}^{(n)}\left(u^{2k}+\mathcal{I}_N\right)\otimes x\quad\forall\,x\in\mathring{\mathfrak{k}}\\
\left(t^{n}-t^{-n}\right)\otimes y  &\mapsto\left(\pm1\right)^{n}\sum_{k=0}^{\left\lfloor (N-1)/2\right\rfloor }a_{2k+1}^{(n)}\left(u^{2k+1}+\mathcal{I}_N\right)\otimes y\quad\forall\,y\in\mathring{\mathfrak{p}}.
\end{align}
\end{cor}

\begin{proof}
$\mathfrak{N}\left(\mathbb{P}_N\right)$ is a quotient of $\mathfrak{N}\left(\mathbb{K}[[u]]\right)$
because $\mathbb{P}_{N}=\mathbb{K}[[u]]\diagup\mathcal{I}_N$
is a quotient of $\mathbb{P}$.
\end{proof}
Next, we want to show that the homomorphisms $\rho_{\pm}$ are injective
but not surjective. Towards this we will need the following fact about
matrices whose entries are obtained from evaluation of polynomials:
\begin{lem}
\label{lem:Linear independence of evaluated polynomails}Let $0\neq p_{1},\dots,p_{n}\in\mathbb{K}[t]$
be linearly independent polynomials. Then there exist $N_{1},\dots,N_{n}\in\mathbb{N}$
such that the matrix $\mathcal{M}\left(N_{1},\dots,N_{n}\right)\coloneqq\left(p_{i}\left(N_{j}\right)\right)_{i,j=1}^{n}$
is regular.
\end{lem}

\begin{proof}
This can be achieved by induction on $n$ and an expansion of the determinant which yields a linear combination of linearly independent polynomials. As a nonzero polynomial $p(t)$ is equal to $0$ for only finitely many $t\in\mathbb{K}$ this can be used to show regularity of $\mathcal{M}\left(N_{1},\dots,N_{n}\right)$ for suitable $N_1,\dots,N_n$.
\end{proof}
\begin{prop}
\label{prop:injectivity}
The homomorphism of Lie algebras $\rho_{\pm}:\mathfrak{k}\left(A\right)\left(\mathbb{K}\right)\rightarrow\mathfrak{N}\left(\mathbb{K}[[u]]\right)$
from proposition \ref{prop:ansatz for hom} is injective. Furthermore
its image does not contain elements in $\mathfrak{N}\left(\mathbb{K}[[u]]\right)=\left(\mathbb{P}^{+}\otimes\mathfrak{k}\right)\oplus\left(\mathbb{P}^{-}\otimes\mathfrak{p}\right)$
whose formal power series contain only finitely many nonzero coefficients.
In particular, the elements $u^{2k+2}\otimes x$ for $x\in\mathring{\mathfrak{k}}$
and $u^{2k+1}\otimes y$ for $y\in\mathring{\mathfrak{p}}$
and $k\geq0$ are not contained in the image of $\mathfrak{k}\left(A\right)\left(\mathbb{K}\right)$
in $\mathfrak{N}\left(\mathbb{K}[[u]]\right)$.
\end{prop}

\begin{rem}
Observe that there is a certain asymmetry. It is possible to map elements from $\mathfrak{k}(A)\left(\mathbb{K}\right)$ to $\mathfrak{N}\left(\mathbb{K}[[u]]\right)$ by allowing formal power series but in the reverse direction it is not possible to define such a map for elements such as  $u^{2k+2}\otimes x$ because the Laurent polynomials do not have a completion that behaves well under the involution $\eta$. One can only complete in one direction, i.e., for an element $\sum_{n\in\mathbb{Z}} c_n u^n\otimes x$ in some completion of $\mathfrak{k}\left(A\right)\left(\mathbb{K}\right)$, $c_n\neq0$ is only possible for either $n>N$ or $n<N$ but not both at the same time. But this does not agree with the demand $c_{-n}=\pm c_n$ unless $c_n\neq0$ for only finitely many $n$.
\end{rem}
\begin{proof}
For a generic element in $\mathfrak{k}\left(A\right)\left(\mathbb{K}\right)$
one can split the analysis into two pieces because elements from $\mathfrak{L}^{+}\otimes\mathring{\mathfrak{k}}$
are mapped to elements which only involve even powers $u^{2k}$ while
the ones from $\mathfrak{L}^{-}\otimes\mathring{\mathfrak{p}}$
are mapped to series involving only odd powers $u^{2k+1}$. Therefore
consider the image of 
\[
\chi\coloneqq\sum_{i=1}^{K}\left(t^{n_{i}}+t^{-n_{i}}\right)\otimes x_{i}
\]
under $\rho$: 
\[
\rho\left(\chi\right)=\sum_{i=1}^{K}\sum_{k\geq0}a_{2k}^{(n_{i})}u^{2k}\otimes x_{i}=\sum_{k\geq0}\sum_{i=1}^{K}a_{2k}^{(n_{i})}u^{2k}\otimes x_{i}=0,
\] 
so that we need $ \sum_{i=1}^{K}a_{2k}^{(n_{i})}x_{i}  =  0$ for all $k\geq 0$.
For a basis $e_{1},\dots,e_{d}$ of $\mathring{\mathfrak{k}}$
and $x_{i}=\sum_{j=1}^{d}c_{i}^{j}e_{j}$ this yields 
\[
\sum_{i=1}^{K}\sum_{j=1}^{d}a_{2k}^{(n_{i})}c_{i}^{j}e_{j}  =  0\ \forall\,k\geq0\ 
\Leftrightarrow\ \sum_{i=1}^{K}a_{2k}^{(n_{i})}c_{i}^{j}  =  0\ \forall\,k\geq0,\ \forall\,j=1,\dots,d.
\]
This way one sees that $\rho(\chi)=0$ admits nontrivial
solutions if and only if the infinite linear system of equations 
\begin{equation}
\sum_{i=1}^{K}a_{2k}^{(n_{i})}z_{i}=0\ \forall\,k\geq0\label{eq:kernel condition 1b}
\end{equation}
 does. The coefficient $a_{2N}^{(n)}$ is given by
the evaluation at $N$ of a nontrivial polynomial $p_{n}\in\mathbb{K}[x]$
of degree $n-1$ according to eq. (\ref{eq:coeffs}) Note that the polynomials $p_{n_{1}},\dots,p_{n_{K}}$
are linearly independent if the $n_{1},\dots,n_{K}$ are pairwise
distinct because then they are of different degree. Consider a subsystem of linear
equations of (\ref{eq:kernel condition 1b}) given by 
\[
\sum_{i=1}^{K}a_{2k}^{(n_{i})}z_{i}=0\quad \forall\,k\in\left\{ N_{1},\dots,N_{K}\right\}\ \Leftrightarrow\ \sum_{i=1}^{K}p_{n_{i}}(k)z_{i}=0\quad \forall\,k\in\left\{ N_{1},\dots,N_{K}\right\}. 
\]

By lemma \ref{lem:Linear independence of evaluated polynomails}
it is possible to choose $N_{1},\dots,N_{K}$ such that the matrix
$\left(p_{n_{i}}\left(N_{j}\right)\right)_{i,j=1}^{K}$ is regular.
Thus, this subsystem of linear equations admits only the trivial solution
and therefore so does (\ref{eq:kernel condition 1b}). Another way
to put this result is that for a basis $\left\{ e_{1},\dots,e_{d}\right\} $
 of $\mathring{\mathfrak{k}}$ the elements of the set
\begin{equation}
\left\{ \rho\left(\left(t^{n}+t^{-n}\right)\otimes e_{i}\right)\ \vert\ n\geq0,\ i=1,\dots,d\right\} \in\mathfrak{N}\left(\mathbb{K}\right)\label{eq:image of L+ otimes k}
\end{equation}
are linearly independent. Since $a_{2N+1}^{(n)}$ is also given by a polynomial  the same argument works for 
\begin{equation}
\left\{ \rho\left(\left(t^{n}-t^{-n}\right)\otimes f_{i}\right)\ \vert\ n\geq1,\ i=1,\dots,D\right\} \in\mathfrak{N}\left(\mathbb{K}\right),\label{eq:image of L- otimes p}
\end{equation}
where $\left\{ f_{1},\dots,f_{D}\right\} $ is a basis of $\mathring{\mathfrak{p}}$.
This shows that $\rho_{\pm}$ is injective.

We next consider the claim that elements $\sum_{k\geq0}b_{2k}u^{2k}\otimes x_{k}$
with only finitely many $b_{2k}\neq0$ are not contained in the image
of $\mathfrak{k}\left(A\right)\left(\mathbb{K}\right)$ in $\mathfrak{N}\left(\mathbb{K}[[u]]\right)$.
Any element in the image of $\mathfrak{k}\left(A\right)\left(\mathbb{K}\right)$
in $\mathfrak{N}\left(\mathbb{K}[[u]]\right)$ can be written as a linear
combination of elements in (\ref{eq:image of L+ otimes k}) and
(\ref{eq:image of L- otimes p}). It suffices to focus on one
set as they are split into even and odd coefficients. For elements
of type (\ref{eq:image of L+ otimes k}) this implies that there
exist $b_{1},\dots,b_{N}\in\mathbb{K}\setminus\{0\}$ such that 
\[
\sum_{j=1}^{N}b_{j}a_{2k}^{(n_{j})}=0\ \forall\,k>0\quad \Leftrightarrow\quad \sum_{j=1}^{N}b_{j}p^{(n_{j})}(k)=0\ \forall\,k>0.
\]
The polynomials $p^{(n_{j})}$ are linearly independent and so their
sum is a polynomial of fixed degree greater than $0$. Then the above
equation is a contradiction to the fact that a nonzero polynomial
can be equal to $0$ only at finitely many points.
\end{proof}

Propositions \Ref{prop:ansatz for hom} and \ref{prop:injectivity} show the first part of theorem \ref{thm:thm_A} concerning $\mf{N}\left(\mathbb{P}\right)$ and we consider the statements concerning $\mf{N}\left(\mathbb{P}_N\right)$ next.
\begin{prop}
\label{prop:rhoN}
The homomorphisms $\rho_{\pm}^{(N)}:\mathfrak{k}\left(A\right)\left(\mathbb{K}\right)\rightarrow\mathfrak{N}\left(\mathbb{P}_N\right)$
are surjective.
\end{prop}
\begin{proof}
Set 
\begin{subequations}
\begin{align}
\mathring{\mathfrak{k}}_{(2k)} & \coloneqq \text{span}_{\mathbb{K}}\left\{ \left(u^{2k}+\mathcal{I}_N\right)\otimes x\ \vert\ x\in\mathring{\mathfrak{k}}\right\} \subset\mathfrak{N}\left(\mathbb{P}_N\right),\label{eq:homogeneous_k}\\
\mathring{\mathfrak{p}}_{(2k+1)} & \coloneqq\text{span}_{\mathbb{K}}\left\{ \left(u^{2k+1}+\mathcal{I}_N\right)\otimes y\ \vert\ y\in\mathring{\mathfrak{p}}\right\} \subset\mathfrak{N}\left(\mathbb{P}_N\right) \label{eq:homogeneous_p}
\end{align}
\end{subequations} 
and note that $\mathfrak{N}\left(\mathbb{P}_N\right)$ decomposes into vector spaces as
\begin{equation}
\mathfrak{N}\left(\mathbb{P}_N\right)\cong\bigoplus_{k=0}^{\left\lfloor N/2\right\rfloor }\mathring{\mathfrak{k}}_{(2k)}\oplus\bigoplus_{k=0}^{\left\lfloor (N-1)/2\right\rfloor }\mathring{\mathfrak{p}}_{(2k+1)}. \label{eq:decomposition of NNK}
\end{equation}
The image of $\rho_{\pm}^{(N)}$ in $\mathfrak{N}\left(\mathbb{P}_N\right)$
is spanned by elements of the form 
\[
\sum_{k=0}^{\left\lfloor N/2\right\rfloor }a_{2k}^{(n)}\left(u^{2k}+\mathcal{I}_N\right)\otimes x\ \forall\,n\in\mathbb{N},\,x\in\mathring{\mathfrak{k}},\quad\sum_{k=0}^{\left\lfloor (N-1)/2\right\rfloor }a_{2k+1}^{(n)}\left(u^{2k+1}+\mathcal{I}_N\right)\otimes y\ \forall\,n\in\mathbb{N}^*,\,y\in\mathring{\mathfrak{p}}.
\]
Since $a_{2k}^{(0)}=2\delta_{k,0}$ one already has $\mathring{\mathfrak{k}}_{(0)}=\left\{ 1\otimes x\ \vert\ x\in\mathring{\mathfrak{k}}\right\}\subset \text{im}\ \rho_\pm^{(N)} $.
With this it is possible to remove the $\mathring{\mathfrak{k}}_{(0)}$-part
from other elements:
\[
\rho_{\pm}^{(N)}\left(x_{(m)}-\frac{1}{2}\left(\pm1\right)^{m}a_{0}^{(m)}x_{(0)}\right)=\sum_{k=1}^{\left\lfloor N/2\right\rfloor }a_{2k}^{(n)}\left(u^{2k}+\mathcal{I}_N\right)\otimes x\ \forall\,n\in\mathbb{N},\,x\in\mathring{\mathfrak{k}}.
\]
By the properties of the Cartan decomposition (cf. \cite[prop. 13.1.10]{HN2012}) one has that if $\mathring{\mathfrak{g}}$ is simple and non-compact then $\mathring{\mathfrak{k}}=\left[\mathring{\mathfrak{p}},\mathring{\mathfrak{p}} \right]$ and that $\mathring{\mathfrak{p}}$ is a simple $\mathring{\mathfrak{k}}$-module. Therefore any element $x\in\mathring{\mathfrak{k}}$
can be written as an iterated commutator $\left[x^{(1)},\left[x^{(2)},\dots,x^{(k)}\right]\right]$
for $x^{(1)},\dots,x^{(k)}\in\mathring{\mathfrak{k}}$ or $\mathring{\mathfrak{p}}$. Choose levels
$n_{1},\dots,n_{k}$ such that $n_{1}+\dots+n_{k}=\left\lfloor N/2\right\rfloor $
and set 
\[
\widetilde{x}^{(i)}_{(n_{i})}\coloneqq x^{(i)}_{(n_{i})}-\frac{1}{2}\left(\pm1\right)^{n_{i}}a_{0}^{(n_{i})}x^{(i)}_{(0)}.
\]
Then 
\[
\rho_{\pm}^{(N)}\left(\left[\widetilde{x}^{(1)}_{(n_{1})},\left[\dots,\widetilde{x}^{(k)}_{(n_{k})}\right]\right]\right)=\left(u^{2\left\lfloor N/2\right\rfloor }+\mathcal{I}_N\right)\otimes x
\]
which shows that
\[
\left\{ \left(u^{2\left\lfloor N/2\right\rfloor }+\mathcal{I}_N\right)\otimes x\ \vert\ x\in\mathring{\mathfrak{k}}\right\} 
\]
 is contained in the image of $\rho_{\pm}^{(N)}$. The same procedure works for $\left(u^{2k+1}+\mathcal{I}_N\right)\otimes\mathring{\mathfrak{p}}$ because $\mathring{\mathfrak{p}}$ is a simple $\mathring{\mathfrak{k}}$-module and therefore there exist iterated commutators here as well. Repeat this process for the
lower levels as now the level $\left(u^{2\left\lfloor N/2\right\rfloor }+\mathcal{I}_N\right)$
can be removed. This shows by induction that each homogeneous space in (\ref{eq:decomposition of NNK}) lies in the image of $\rho_\pm^{(N)}$ which concludes the proof.
\end{proof}

With this, we have proven all parts of theorem \ref{thm:thm_A}. We conclude this section with some results on the structure of $\mf{N}(\mathbb{P}_N)$ and the kernel of $\rho_{\pm}^{(N)}$.

\begin{prop}\label{prop:strucN}
Let $\mathring{\mathfrak{k}}_{(2k)}$ and $\mathring{\mathfrak{p}}_{(2k+1)}$ be as in (\ref{eq:homogeneous_k}) and (\ref{eq:homogeneous_p}) and denote by $\mathfrak{z}\big(\mathring{\mathfrak{k}}_{(0)}\big)$ the center\footnote{Note that for a generalised Dynkin diagram $A$ of untwisted affine
type, the only cases when $\mathfrak{z}(\mathring{\mathfrak{k}}_{(0)})$
is nontrivial are $A=C_{l}^{(1)}$ and $A=A_{1}^{(1)}$. For $\mathring{\mathfrak{k}}_{(0)}=\mathfrak{k}\left(C_{l}\right)$
one has $\mathfrak{k}\left(C_{l}\right)\cong\mathfrak{u}_{l}$ which
contains a nontrivial center, whereas for $\mathfrak{k}\left(A_{1}\right)\cong\mathbb{R}$
the center is already all of $\mathring{\mathfrak{k}}_{(0)}$.} of $\mathring{\mathfrak{k}}_{(0)}$, then 
\[
\mathfrak{J}_{(N)}\coloneqq\mathfrak{z}\big(\mathring{\mathfrak{k}}_{(0)}\big)\oplus\bigoplus_{k=1}^{\left\lfloor N/2\right\rfloor }\mathring{\mathfrak{k}}_{(2k)}\oplus\bigoplus_{k=0}^{\left\lfloor (N-1)/2\right\rfloor }\mathring{\mathfrak{p}}_{(2k+1)}\subset\mathfrak{N}\left(\mathbb{P}_N\right)
\]
is the radical of $\mathfrak{N}\left(\mathbb{P}_N\right)$, i.e.,
the unique maximal solvable ideal in $\mathfrak{N}\left(\mathbb{P}_N\right)$.
Hence, the Levi decomposition of $\mathfrak{N}\left(\mathbb{P}_N\right)$
is given by 
\[
\mathfrak{N}\left(\mathbb{P}_N\right)\cong\left[\mathring{\mathfrak{k}}_{(0)},\mathring{\mathfrak{k}}_{(0)}\right]\ltimes\mathfrak{J}_{(N)}.
\]

\end{prop}

\begin{proof}
The $\mathbb{N}$-graded structure of $\mathfrak{N}\left(\mathbb{P}_N\right)$
is given by its decomposition into vector spaces (\ref{eq:decomposition of NNK}). By the gradation one deduces 
\[
\left[\mathring{\mathfrak{p}}_{(2k-1)},\mathring{\mathfrak{p}}_{(2\ell-1)}\right]\subseteq\mathring{\mathfrak{k}}_{(2k+2\ell-2)},\ \left[\mathring{\mathfrak{k}}_{(2k)},\mathring{\mathfrak{p}}_{(2\ell-1)}\right]\subseteq\mathring{\mathfrak{p}}_{(2k+2\ell-1)},\ \left[\mathring{\mathfrak{k}}_{(2k)},\mathring{\mathfrak{k}}_{(2\ell)}\right]\subseteq\mathring{\mathfrak{k}}_{(2k+2\ell)},
\]
where it is understood that $\mathring{\mathfrak{k}}_{(2k)}=\{0\}=\mathring{\mathfrak{p}}_{(2\ell+1)}$
for $k>\left\lfloor N/2\right\rfloor $ and $\ell>\left\lfloor (N-1)/2\right\rfloor $.
From this it follows that $\mathfrak{J}_{(N)}$ is an ideal as in
particular 
\[
\left[\mathring{\mathfrak{k}}_{(0)},\mathring{\mathfrak{k}}_{(2\ell)}\right]\subseteq\mathring{\mathfrak{k}}_{(2\ell)}\ \text{and }\left[\mathring{\mathfrak{k}}_{(0)},\mathring{\mathfrak{p}}_{(2\ell-1)}\right]\subseteq\mathring{\mathfrak{p}}_{(2\ell-1)}.
\]
Since $\left[\mathfrak{z}\big(\mathring{\mathfrak{k}}_{(0)}\big),\mathfrak{z}\big(\mathring{\mathfrak{k}}_{(0)}\big)\right]=\{0\}$
the lowest degree in $\left[\mathfrak{J}_{(N)},\mathfrak{J}_{(N)}\right]$
is $1$ and therefore $\mathfrak{J}_{(N)}$ is solvable because $\mathring{\mathfrak{k}}_{(2k)}=\{0\}=\mathring{\mathfrak{p}}_{(2l+1)}$
for $k>\left\lfloor N/2\right\rfloor $ and $\ell>\left\lfloor (N-1)/2\right\rfloor $.
Consider the ideal generated by $\mathfrak{J}_{(N)}+x$ for $0\neq x\in\left[\mathring{\mathfrak{k}}_{(0)},\mathring{\mathfrak{k}}_{(0)}\right]$.
As $\left[\mathring{\mathfrak{k}}_{(0)},\mathring{\mathfrak{k}}_{(0)}\right]$
is semi-simple so is the ideal $\mathfrak{j}_{0}$ in $\left[\mathring{\mathfrak{k}}_{(0)},\mathring{\mathfrak{k}}_{(0)}\right]$
generated by $x$. Since ideals of semisimple Lie algebras are semisimple,
$\mathfrak{j}_{0}$ is also perfect. Thus, the upper derived series
$\mathfrak{j}_{0}^{(n+1)}\coloneqq\left[\mathfrak{j}_{0}^{(n)},\mathfrak{j}_{0}^{(n)}\right]$
becomes constant at $\mathfrak{j}_{0}\coloneqq\left[\mathfrak{j}_{0},\mathfrak{j}_{0}\right]$
and therefore the upper derived series $\mathfrak{J}_{(N)}^{(n+1)}\coloneqq\left[\mathfrak{J}_{(N)}^{(n)},\mathfrak{J}_{(N)}^{(n)}\right]$
will always contain $\mathfrak{j}_{0}$. Thus, $\mathfrak{J}_{(N)}$
is a maximal solvable ideal and therefore by definition the radical
of $\mathfrak{N}\left(\mathbb{P}_N\right)$.
\end{proof}

\begin{rem}
We call $\mf{N}(\mathbb{K}[[u]])$ a parabolic Lie algebra since it is the inverse limit of the parabolic Lie algebra $\mf{N}(\mathbb{P}_N)$. By proposition~\ref{prop:ansatz for hom} and corollary~\ref{cor:truncated hom}, we can construct representations of the involutory subalgebra $\mf{k}$ by considering representations of $\mf{N}(\mathbb{K}[[u]])$ or $\mf{N}(\mathbb{P}_N)$, respectively.
\end{rem}

\begin{rem}
By a consequence of Lie's theorem one has that every simple representation of $\mathfrak{N}\left(\mathbb{P}_N\right)$ over a complex vector space is given by the tensor
product of a simple representation of $\mathfrak{N}\left(\mathbb{P}_N\right)\diagup \mathrm{rad}\left(\mathfrak{N}\left(\mathbb{P}_N\right)\right)\cong [\khor,\khor]$
with a one-dimensional representation of $\mathfrak{N}\left(\mathbb{P}_N\right)$.
Therefore, the simple representations of $\mathfrak{k}\left(A\right)\left(\mathbb{K}\right)$
that factor through $\mathfrak{N}\left(\mathbb{P}_N\right)$ are
essentially the simple representations of $\mathring{\mathfrak{k}}\left(\mathbb{K}\right)$. These correspond to truncation at $N=0$.
\end{rem}

The kernels $\ker\rho_{\pm}^{(N)}$ are described by linear systems
of equations as the following proposition shows:
\begin{prop}
Let $\ker\rho_{\pm}^{(N)}$ be the kernel of the homomorphism described
in cor. \ref{cor:truncated hom}. Then with 
\[
x_{(m)}\coloneqq\left(t^{m}+t^{-m}\right)\otimes x\ \forall\,x\in\mathring{\mathfrak{k}},\qquad y_{(m)}\coloneqq\left(t^{m}-t^{-m}\right)\otimes y\ \forall\,y\in\mathring{\mathfrak{p}}
\]
one has
\begin{align*}
\ker\rho_{\pm}^{(N)} & =  \mathrm{span}\left\{ \sum_{i=1}^{M}\left(\pm1\right)^{m_{i}}b_{i}x_{(m_{i})}\, \middle| \,\sum_{i=1}^{M}\left(\pm1\right)^{m_{i}}b_{i}a_{2k}^{(m_{i})}=0\ \forall\,k{=}0,{\ldots},\left\lfloor N/2\right\rfloor ,\ \forall\,x{\in}\mathring{\mathfrak{k}}\right\} \\
 &  \quad\oplus \mathrm{span}\left\{ \sum_{i=1}^{M}\left(\pm1\right)^{m_{i}}b_{i}y_{(m_{i})} \,\middle|\, \sum_{i=1}^{M}\left(\pm1\right)^{m_{i}}b_{i}a_{2k+1}^{(m_{i})}=0\ \forall\,k{=}0,{\ldots},\left\lfloor (N{-}1)/2\right\rfloor ,\ \forall\,y{\in}\mathring{\mathfrak{p}}\right\} \\
\ker\rho_{\pm}^{(N)} & \supset  \ker\rho_{\pm}^{(N+1)}.
\end{align*}
\end{prop}

\begin{proof}
In general one has for $x\in\mathring{\mathfrak{k}}$, $y\in\mathring{\mathfrak{p}}$
that
\[
\rho_{\pm}^{(N)}\left(x_{(m)}\right)=\left(\pm1\right)^{m}\sum_{k=0}^{\left\lfloor N/2\right\rfloor }a_{2k}^{(m)}u^{2k}\otimes x,\quad\rho_{\pm}^{(N)}\left(y_{(m)}\right)=\left(\pm1\right)^{m}\sum_{k=0}^{\left\lfloor (N-1)/2\right\rfloor }a_{2k+1}^{(m)}u^{2k+1}\otimes y
\]
and one computes for $\sum_{i=1}^{M}\left(\pm1\right)^{m_{i}}b_{i}x_{(m_{i})}$
that
\begin{eqnarray*}
\rho_{\pm}^{(N)}\left(\sum_{i=1}^{M}\left(\pm1\right)^{m_{i}}b_{i}x_{(m_{i})}\right) & = & \sum_{i=1}^{M}\left(\pm1\right)^{m_{i}}b_{i}\sum_{k=0}^{\left\lfloor N/2\right\rfloor }a_{2k}^{(m_{i})}u^{2k}\otimes x\\
 & = & \sum_{k=0}^{\left\lfloor N/2\right\rfloor }\left(\sum_{i=1}^{M}\left(\pm1\right)^{m_{i}}b_{i}a_{2k}^{(m_{i})}\right)u^{2k}\otimes x=0
\end{eqnarray*}
if and only if 
\[
\sum_{i=1}^{M}\left(\pm1\right)^{m_{i}}b_{i}a_{2k}^{(m_{i})}=0\ \forall\,k=0,\dots,\left\lfloor N/2\right\rfloor .
\]
Similarly one deduces that 
\[
\rho_{\pm}^{(N)}\left(\sum_{i=1}^{M}\left(\pm1\right)^{m_{i}}b_{i}y_{(m_{i})}\right)=0\ \Leftrightarrow\ \sum_{i=1}^{M}\left(\pm1\right)^{m_{i}}b_{i}a_{2k+1}^{(m_{i})}=0\ \forall\,k=0,\dots,\left\lfloor (N-1)/2\right\rfloor .
\]
These equations remain unaltered by changing $N\mapsto N+N_{0}$,
$N_{0}\in\mathbb{N}$, there only appear additional equations to be
satisfied. This shows that 
\[
\ker\rho_{\pm}^{(N)}\supset\ker\rho_{\pm}^{(N+1)}.
\]
\end{proof}
\begin{rem}
Specialised to $\mf{e}_9$ the ideals $\ker\rho_{\pm}^{(0)}$ in $\mf{k}\left(\mf{e}_9\right)(\mathbb{K})$ coincide with the Dirac ideals of \cite{Kleinschmidt:2006dy}. Furthermore, these ideals can be shown to be principal ideals.
\end{rem}

\section{Representations from quotients of induced representations}
\label{sec:verm}

In this section, we consider representations of $\mf{k}$ that are constructed using induced representations of $\mf{N}\left(\mathbb{K}[u]\right)$,  the model of $\mf{k}$ over the polynomial ring which will be defined below. Since representations of $\mf{N}\left(\mathbb{K}[u]\right)$ in general do not provide representations of $\mf{k}$, we will then describe how to recover representations of $\mf{k}$.

Construct the Lie algebra $\mf{N}\left(\mathbb{K}[u]\right)$  in the same way as $\mf{N}(\mathbb{K}[[u]])$ in~\eqref{eq:Ndef}  but replace $\mathbb{K}[[u]]$ by the ring of polynomials $\mathbb{K}[u]$. As mentioned in remark~\ref{rmk:ser}, $\mf{N}\left(\mathbb{K}[u]\right)$ is $\mathbb{N}$-graded  and its graded decomposition is given by

\begin{align}
\label{eq:Ndec}
\mf{N}\left(\mathbb{K}[u]\right) = \bigoplus_{k=0}^\infty  \mf{N}^{(k)}\,,
\end{align}
where
\begin{align*}
\mf{N}^{(k)} = \left\{ \begin{array}{cl}
\mathrm{span}_{\mathbb{K}} \left\{ u^{2n} \otimes \khor \right\} , & \text{$k=2n$ even},\\
\mathrm{span}_{\mathbb{K}} \left\{ u^{2n+1} \otimes \phor \right\} , & \text{$k=2n+1$ odd}.\\
\end{array}\right.
\end{align*}

We believe that there does not exist any non-trivial Lie algebra homomorphism from $\mf{k}$ into $\mf{N}\left(\mathbb{K}[u]\right)$. Such homomorphisms, similar to those of proposition~\ref{prop:rhoN},  only exist when we quotient by $\mathcal{I}_N$ defined in~\eqref{eq:idN} which can also be thought of as an ideal of $\mathbb{K}[u]$.
 The structure of $\mf{N}\left(\mathbb{K}[u]\right)$ is analogous to that of $\mf{N}(\mathbb{P}_N)$ described in proposition~\ref{prop:strucN}. In particular, the (now no longer solvable) ideal $\mf{N}_+$ is given by 
\begin{align*}
\mf{N}_+  = \mf{z}(\khor) \oplus \bigoplus_{k>0} \mf{N}^{(k)}\,,
\end{align*}
where $\mf{z}(\khor)$ denotes the center of $\khor$. The ideal $\mf{N}_{+} $ inherits the $\mathbb{N}$-grading from $\mf{N}\left(\mathbb{K}[u]\right)$. We shall in the following assume that $\mf{z}(\khor)=0$ for simplicity. All statements can be generalised straight-forwardly. 

The universal enveloping algebras $\mathcal{U}(\mf{N}_{+})$ and $\mathcal{U}(\mf{N}\left(\mathbb{K}[u]\right))$ also inherit the $\mathbb{N}$-gradation from the Lie algebra and the Poincar\'e--Birkhoff--Witt theorem provides a basis of them. One has
\begin{align}
\mathcal{U}(\mf{N}_{+}) = \bigoplus_{\ell=0}^\infty\, \mathcal{U}_\ell,
\end{align}
where $\mathcal{U}_\ell$ denotes the (ordered) words in the tensor algebra of degree $\ell$ with respect to the $\mathbb{N}$-grading, where $x\leq y$ if $\text{deg}(x)\leq \text{deg}(y)$. For the first few levels, this means
\begin{align*}
\label{eq:vermbas}
\mathcal{U}_0 &= 1\,,& \quad
\mathcal{U}_1 &= \mf{N}^{(1)} \,,\\
\mathcal{U}_2 &= \mathrm{Sym}^2\left(\mf{N}^{(1)}\right) \oplus \mf{N}^{(2)}\,,& \quad
\mathcal{U}_3 &= \mathrm{Sym}^3\left(\mf{N}^{(1)}\right) \oplus \mf{N}^{(1)} \otimes \mf{N}^{(2)} \oplus \mf{N}^{(3)}\,.\nn
\end{align*}
The $\mf{N}^{(k)}$ are $\khor$-modules because the adjoint action of $\khor$ preserves the degree. We will make these expressions more explicit in the example in section~\ref{sec:Specialisation}. For the full universal enveloping algebra one has 
\begin{align}
\mathcal{U}(\mf{N}\left(\mathbb{K}[u]\right)) \cong \mathcal{U}(\khor) \otimes_{\mathbb{K}} \mathcal{U}(\mf{N}_{+})
\end{align}
as a tensor product of $\mathbb{K}$-vector
spaces. In terms of the multiplication $\cdot$ in $\mathcal{U}\left(\mf{N}(\mathbb{K}[u])\right)$
the above decomposition is better written as 
\begin{align}
\mathcal{U}\left(\mf{N}\left(\mathbb{K}[u]\right)\right)=\mathcal{U}(\khor)\cdot\mathcal{U}\left(\mf{N}_{+}\right)\label{eq:UEA-decomposition}
\end{align}
This fact is due to the PBW-theorem applied to a suitably chosen order
that is such that elements of degree $0$ appear to the left in the
basis that is provided by the PBW-theorem. Recall that all elements
of $\khor=\mf{N}^{(0)}\subset\mf{N}\left(\mathbb{K}[u]\right)$
have degree $0$.

\medskip

Now consider a finite-dimensional $\khor$-module $\tV_0$ which we view as a left $\mathcal{U}\big(\khor\big)$-module. As $\mathcal{U}\big(\mf{N}\left(\mathbb{K}[u]\right)\big)$ allows the structure of a right $\mathcal{U}\big(\khor\big)$-module we build the induced 
$\mf{N}\left(\mathbb{K}[u]\right)$-module via the tensor product
\begin{align}
\label{eq:induced_V}
\tV\coloneqq\mathcal{U}\left(\mf{N}\left(\mathbb{K}[u]\right)\right)\otimes_{\mathcal{U}(\khor)}\tV_0,
\end{align}
where the tensor product is defined with $\mathcal{U}\left(\mf{N}\left(\mathbb{K}[u]\right)\right)$
 as a right $\mathcal{U}\big(\khor\big)$-module. Explicitly one has 
\begin{align*}
a\otimes\left(x\cdot v\right)=ax\otimes v
\end{align*}
for $v\in V$, $x\in\mathcal{U}\big(\khor\big)$, $a\in\mathcal{U}\left(\mf{N}\left(\mathbb{K}[u]\right)\right)$. 
$\tV$, however, is viewed as a left $\mathcal{U}\left(\mf{N}\left(\mathbb{K}[u]\right)\right)$-module, i.e. for $a,b\in\mathcal{U}\left(\mf{N}\left(\mathbb{K}[u]\right)\right)$,
$v\in V$ one has 
\begin{align*}
a\cdot\left(b\otimes v\right)=\left(a\cdot b\right)\otimes v.
\end{align*}

\begin{lem}
\label{lem:induced_V}The induced module $\tV$ inherits the $\mathbb{N}$-grading
\begin{align}
\label{eq:V}
\tV = \bigoplus_{i=0}^\infty \tV_i
\end{align}
from $\mathcal{U}\left(\mf{N}\left(\mathbb{K}[u]\right)\right)$ and is an infinite-dimensional representation of $\mf{N}\left(\mathbb{K}[u]\right)$.  $\tV$ is generated by the action of $\mathcal{U}(\mf{N}_{+})$ on $1\otimes \tV_0$. Furthermore, a $\mathbb{K}$-basis of $\tV$ is given by the Kronecker product of a PBW-basis of $\mathcal{U}\left(\mf{N}_{+}\right)$ with a basis of $\tV_0$.
\end{lem}

\begin{proof}
The inherited gradation on $\tV$ is given by assigning $\mathrm{deg}(u\otimes v)=\mathrm{deg}(u)$. It is infinite-dimensional because $\mathcal{U}\left(\mf{N}_{+}\right)$ is. In order to determine a $\mathbb{K}$-basis of (\ref{eq:induced_V}), consider it as a $\khor$-module
and show that it is isomorphic to the $\mathbb{K}$-tensor product
of $\mathcal{U}\left(\mf{N}_{+}\right)$ and $\tV_0$. Recall that $\khor$
acts on $\mf{N}_{+}$ via the adjoint action from $\mf{N}\left(\mathbb{K}[u]\right)$.
Due to the above decomposition (\ref{eq:UEA-decomposition}) one has for $x\in\mathcal{U}\big(\khor\big)$,
$y\in\mathcal{U}\left(\mf{N}_{+}\right)$ and $v\in V$ that
\begin{align}
xy\otimes v=\left(yx+\left[x,y\right]\right)\otimes v=y\otimes v'+y'\otimes v\label{eq:shifting}
\end{align}
with $y'=\left[x,y\right]\in\mathcal{U}(\mathfrak{N}_{+})$ and $v'=x.v\in V$.
We want to view (\ref{eq:induced_V}) as a $\mathbb{K}$-tensor
product modulo an equivalence relation. The equivalence relation on
$\mathcal{U}\big(\khor\big)\otimes_{\mathbb{K}}\mathcal{U}\left(\mf{N}_{+}\right)\otimes_{\mathbb{K}}V$
induced by $\otimes_{\mathcal{U}\left(\khor\right)}$
now is the multilinear extension of
\begin{align}
x\otimes_{\mathbb{K}}y\otimes_{\mathbb{K}}v\sim1\otimes_{\mathbb{K}}y\otimes_{\mathbb{K}}v'+1\otimes_{\mathbb{K}}y'\otimes_{\mathbb{K}}v,\label{eq:equivalence relation}
\end{align}
where $y'=\left[x,y\right]\in\mathcal{U}(\mf{N}_{+})$, $v'=x.v\in V$ and
\begin{align*}
\mathcal{U}\left(\mf{N}\left(\mathbb{K}[u]\right)\right)\otimes_{\mathcal{U}\left(\khor\right)}V\cong\mathcal{U}\big(\khor\big)\otimes_{\mathbb{K}}\mathcal{U}\left(\mf{N}_{+}\right)\otimes_{\mathbb{K}}V\diagup\sim
\end{align*}
as $\mathbb{K}$-vector spaces. The original equivalence relation in $\mathcal{U}\left(\mf{N}\left(\mathbb{K}[u]\right)\right)\otimes_{\mathbb{K}}V$
is
\begin{align*}
yx\otimes_{\mathbb{K}}v\sim y\otimes_{\mathbb{K}}xv\ \forall\,x\in\mathcal{U}\big(\khor\big)
\end{align*}
but decomposing $yx=xy-\left[x,y\right]$ according to (\ref{eq:UEA-decomposition})
leads to the formulation (\ref{eq:equivalence relation}) of $\sim$.
Eq. (\ref{eq:equivalence relation}) shows that each element $x\otimes_{\mathbb{K}}y\otimes_{\mathbb{K}}v$
is $\sim$-equivalent to an element of $1\otimes_{\mathbb{K}}\mathcal{U}\left(\mf{N}_{+}\right)\otimes_{\mathbb{K}}V\cong\mathcal{U}\left(\mf{N}_{+}\right)\otimes_{\mathbb{K}}V$
where the isomorphism is as $\mathbb{K}$-vector spaces. Since
\begin{align*}
1\otimes_{\mathbb{K}}\mathcal{U}\left(\mf{N}_{+}\right)\otimes_{\mathbb{K}}V\subset\mathcal{U}\big(\khor\big)\otimes_{\mathbb{K}}\mathcal{U}\left(\mf{N}_{+}\right)\otimes_{\mathbb{K}}V
\end{align*}
one deduces that the elements of $\mathcal{U}\big(\khor\big)\otimes_{\mathbb{K}}\mathcal{U}\left(\mf{N}_{+}\right)\otimes_{\mathbb{K}}V\diagup\sim$
are in 1-1-correspondence with the elements of $\mathcal{U}\left(\mf{N}_{+}\right)\otimes_{\mathbb{K}}V$
and hence,

\begin{align}
\mathcal{U}\left(\mf{N}\left(\mathbb{K}[u]\right)\right)\otimes_{\mathcal{U}\left(\khor\right)}V\cong\mathcal{U}\left(\mf{N}_{+}\right)\otimes_{\mathbb{K}}V\label{eq:K-decomposition}
\end{align}
as $\mathbb{K}$-vector spaces. Now $\mathcal{U}\big(\khor\big)$
acts on $\mathcal{U}\left(\mf{N}_{+}\right)\otimes_{\mathcal{U}\left(\khor\right)}V$
via left-multiplication but as a result of (\ref{eq:shifting}) one
has that
\begin{align*}
xy\otimes v=\left[x,y\right]\otimes v+y\otimes xv.
\end{align*}
This action is identical to the action of $\mathcal{U}\big(\khor\big)$
on the $\mathbb{K}$-tensor product of $\khor$-modules
$\mathcal{U}\left(\mf{N}_{+}\right)$ and $V$. Thus, one finds
\begin{align*}
\mathcal{U}\left(\mathfrak{N}\left(\mathbb{K}[u]\right)\right)\otimes_{\mathcal{U}\left(\mathring{\mathfrak{k}}\right)}V\cong\mathcal{U}\left(\mathfrak{N}_{+}\right)\otimes_{\mathbb{K}}V
\end{align*}
as $\khor$-modules. The Kronecker basis of this tensor product also provides a basis for the   $\mathcal{U}\big(\mf{N}\big(\mathbb{K}[u]\big)\!\big)$-module $\tV$ because $\mathcal{U}\left(\mf{N}_{+}\right)\cap\mathcal{U}\big(\khor\big)=\mathbb{K}\cdot1$.

\end{proof}

Due to the $\mathbb{N}$-graded structure on $\mf{N}\left(\mathbb{K}[u]\right)$, $\tV$ has many invariant subspaces. In particular, any
\begin{align}
\label{eq:quotN}
\tV^{(N)} = \bigoplus_{i=N}^\infty \tV_i
\end{align}
for $N>0$ is a proper invariant subspace of $\tV$. The quotient representation 
\begin{align}
\label{quot1}
\tV / \tV^{(N)}
\end{align}
is by construction a finite-dimensional representation of $\mf{N}\left(\mathbb{K}[u]\right)$ for any fixed choice $N$. 
We can now prove the first part of theorem \ref{thm:thm_B}: 
\begin{prop}
\label{prop:truncated_reps}The quotient $\tV / \tV^{(N)}$ defined in~\eqref{quot1} is a finite-dimensional module of $\mf{N}(\mathbb{P}_N)$ and therefore, by propositions~\ref{prop:ansatz for hom} and ~\ref{prop:rhoN}, a representation of $\mf{k}$.
\end{prop}

\begin{proof}
This statement follows from the fact that $\mathcal{I}_N\otimes (\khor \oplus \phor)\cap \mf{N}\left(\mathbb{K}[u]\right)$ acts trivially on $\tV / \tV^{(N)}$ by the $\mathbb{N}$-grading and therefore $\tV / \tV^{(N)}$ is a representation of the corresponding quotient. Since the quotients $\mathbb{K}[u]\diagup \mathcal{I}_N$ and $\mathbb{K}[[u]]\diagup \mathcal{I}_N$ are isomorphic, we deduce that $\tV / \tV^{(N)}$ is a finite-dimensional representation of $\mf{N}(\mathbb{P}_N)$ defined in~\eqref{eq:NNdef} and therefore can be pulled back to a representation of $\mf{k}$ by proposition~\ref{prop:rhoN}.
\end{proof}

\begin{rem}
The quotient algebra acts non-trivially on $\tV / \tV^{(k)}$ and is given by all $\mf{k}$ generators of degree at most $k$. While all other elements of  the parabolic model $\mf{N}\left(\mathbb{K}[u]\right)$ act trivially, we deduce from proposition~\ref{prop:ansatz for hom} that
{\em infinitely} many generators of $\mf{k}$  act non-trivially.
\end{rem}

Other invariant subspaces of $\tV$ can be considered by taking any vector $w\in \tV$ and considering the subrepresentation $\tW\subset \tV$ that it generates under the action of $\mf{N}\left(\mathbb{K}[u]\right)$. A natural choice would be to select some irreducible $\khor$ representation $\tW_k$ within one of the $\tV_k$. Clearly, the submodule $\tW$ generated by $\tW_k$ is a subspace of $\tV^{(k)}$. In general, it can be of arbitrary co-dimension in it. The representation of $\mf{N}\left(\mathbb{K}[u]\right)$ we are interested in then is the quotient
\begin{align}
\label{quot2}
\tV/\tW\,.
\end{align}
If this quotient is finite-dimensional it can again be pulled back to a representation of $\mf{k}$ since it is a quotient of one of the spaces $\tV / \tV^{(N)}$ described above. 
Analysing~\eqref{quot2} in general is more complicated 
than in the case~\eqref{quot1}. To illustrate this point, we shall analyse an explicit example for $\mf{e}_9$ in the next section~\ref{sec:Specialisation}. But first we prove the second part of theorem \Ref{thm:thm_B}.

\begin{prop}\label{prop:rf}
Recall the graded decomposition (\ref{eq:V}) of $\mf{V}$ and consider its formal completion 
\begin{align}
\overline{\mf{V}}\coloneqq\left\{ \left(v_{i}\right)\ \vert\ v_{i}\in \mf{V}_{i}\right\} .
\end{align}
On $\overline{\mf{V}}$ there exists a natural, faithful $\mf{N}\left(\mathbb{K}[[u]]\right)$-action. Furthermore, $\overline{\mf{V}}$ is the inverse limit of the modules of proposition \ref{prop:truncated_reps} and it is spinorial as a $\mf{k}$-representation if the $\khor$-representation $\tV_0$ that induces $\tV$ is. Also, faithfulness shows that $\mf{k}$ is residually finite-dimensional, that is, for each non-trivial element $x\in\mf{k}$ there exists a finite-dimensional representation on which $x$ acts non-trivially.\footnote{We believe this property to be known already for affine $\mf{g}(A)$, although we were unable to find a reference. Hence, our proposition just reproves residual finiteness of $\mf{k}$ as a byproduct.}
\end{prop}
\begin{proof}
The natural action is given by 
\begin{align}
\label{eq:action on formal V}
v\mapsto x\cdot v\quad \text{with} \quad (x\cdot v)_i = \sum_{j=0}^{i}x_{j} \cdot v_{i-j} \quad \text{for $i\in \mathbb{N}$}
\end{align}
for $x=\left(x_{i}\right)\in\mf{N}\left(\mathbb{K}[[u]]\right)$
and $v=\left(v_{i}\right)\in\overline{\mf{V}}$. Since only $x_i$ and $v_{N-i}$ enter at degree $N$ this action commutes with the projection to $\overline{\mf{V}}\diagup\overline{\mf{V}}^{(i)}$, where $\overline{\mf{V}}^{(i)}$ denotes the formal completion of $\mf{V}^{(i)}$. As $\overline{\mf{V}}\diagup\overline{\mf{V}}^{(i)}\cong\mf{V}\diagup\mf{V}^{(i)}$, one can check that $\overline{\mf{V}}$ is the inverse limit of the finite-dimensional representations $\mf{V}\diagup\mf{V}^{(i)}$ by also establishing that the elements in $\overline{\mf{V}}$ and the inverse limit of the $\mf{V}\diagup\mf{V}^{(i)}$ are in 1-1-correspondence. 
The representation is faithful , because for $\left(x_i\right)\in\mf{N}\left(\mathbb{K}[[u]]\right)$ one can pick $\left(v_0,0,0,\dots\right)\in\overline{\tV}$ with $0\neq v_0\in\tV_0$. Since $\mf{N}^{(k)}$ is part of the PBW-basis of $\mathcal{U}\left(\mf{N}\left(\mathbb{K}[u]\right)\right)$, there always exists a $k>0$ and $0\neq v_0$ such that $x_k.v_0\neq0\in\overline{\tV}$ according to lemma \ref{lem:induced_V}. The only exception is when $\left(x_i\right)\in\mf{N}\left(\mathbb{K}[[u]]\right)$ is such that only $x_0\neq0$. Even if $\tV_0$ is the trivial module, the action of this element is nontrivial because $\mathcal{U}\left(\mf{N}_{+}\right)$ is a faithful $\khor$-module. Spinoriality of $\tV_0$ implies that of $\tV$ because $\khor$ maps to $\mf{N}^{(0)}\subset \mf{N}\left(\mathbb{K}[[u]]\right)$.
\end{proof}

\section{\texorpdfstring{A detailed example: $\mf{k}(\mf{e}_9)$}{A detailed example: k(e9)}}
\label{sec:Specialisation}

In this section, we illustrate the general considerations from the previous sections in the case of $\mf{k}(\mf{e}_9)$ where $\mf{g}=\mf{e}_9 \equiv \mf{e}_8^{(1)}$ denotes the non-twisted 
affine extension of the split real exceptional Lie algebra $\ghor=\mf{e}_8$ (also denoted
as $\mf{e}_{8(8)}$). The algebra 
$\mf{k}(\mf{e}_9)$ enters in the fermionic sector of maximal supergravity in $D=2$ dimensions where unfaithful representations have been found previously~\cite{Nicolai:2004nv} and is therefore of particular interest in physics.

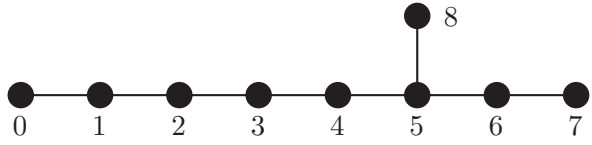
\begin{figure}[t!]
\centering
\begin{picture}(230,50)
\thicklines
\multiput(10,10)(30,0){8}{\circle*{10}}
\put(10,10){\line(1,0){210}}
\put(160,40){\circle*{10}}
\put(160,10){\line(0,1){30}}
\put(7,-5){$0$}
\put(37,-5){$1$}
\put(67,-5){$2$}
\put(97,-5){$3$}
\put(127,-5){$4$}
\put(157,-5){$5$}
\put(187,-5){$6$}
\put(217,-5){$7$}
\put(170,36){$8$}
\end{picture}
\caption{\label{fig:e9dynk}{\sl Dynkin diagram of $E_{9}$ with nodes labelled.}}
\end{figure}

\subsection{\texorpdfstring{The affine algebra $\mf{e}_9$}{The affine algebra e9}}

The Cartan decomposition of $\mf{e}_8$ is
\begin{align}
\label{e8}
\mf{e}_8 = \khor \oplus \phor  \quad\text{with}\quad \khor\cong\mf{so}(16) 
\end{align}
and $\phor$ is a $128$-dimensional irreducible spinor representation of $\mf{so}(16)$. To explicitly parametrise the symmetric space decomposition we introduce the following adapted basis of generators
\begin{align}
\label{e8bas}
\mf{so}(16) &= \big\langle X^{IJ} \,\big|\, X^{IJ}=-X^{JI}\,,\quad I,J=1,\ldots 16 \big\rangle\,,\nn\\
\phor &= \big\langle Y^A \,\big|\, A=1,\ldots, 128\big\rangle\,,
\end{align}
so that the $X^{IJ}$ are the exterior square of the defining $16$-dimensional representation of $\mf{so}(16)$ and the $Y^A$ are in an irreducible spinor representation. The conjugate spinor representation will be denoted with dotted indices $\dot{A}$ and will play a role in the construction of $\mf{k}(\mf{e}_9)$ representations below.

In order to write the $\mf{e}_8$ commutation relations in the basis~\eqref{e8bas}, we utilise $\mf{so}(16)$ gamma matrices $\Gamma^I_{A\dot{A}}$ where $A\dot{A}$ denote the matrix components of degree one elements in the Clifford algebra of $16$-dimensional Euclidean space. These are real 128-by-128 matrices satisfying the  Clifford multiplication law
\begin{align}
\label{gammas}
\Gamma^I_{A\dot{A}} \Gamma^J_{B\dot{A}} + \Gamma^J_{A\dot{A}} \Gamma^I_{B\dot{A}} = 2 \delta^{IJ} \delta_{AB} 
\end{align}
where repeated indices are summed over (a convention we shall employ throughout this section) and $\delta^{IJ}$ and $\delta_{AB}$ are the Euclidean $\mf{so}(16)$-invariant bilinear forms in the representation spaces. $k$-fold antisymmetric products of gamma matrices are written as $\Gamma^{I_1\ldots I_k}$ normalised with a $1/k!$ times the signed sum over permutations of $I_1,\ldots , I_k$, e.g.
\begin{align}
\Gamma^{I_1I_2}_{AB} &= \frac1{2!} \left( \Gamma^{I_1}_{A\dot{A}} \Gamma^{I_2}_{B\dot{A}} - \Gamma^{I_2}_{A\dot{A}} \Gamma^{I_1}_{B\dot{A}}\right) \,,\nn\\
\Gamma^{I_1I_2}_{\dot{A}\dot{B}} &= \frac1{2!} \left( \Gamma^{I_1}_{A\dot{A}} \Gamma^{I_2}_{A\dot{B}} - \Gamma^{I_2}_{A\dot{A}} \Gamma^{I_1}_{A\dot{B}}\right) \,,\nn\\
\Gamma^{I_1I_2I_3}_{A\dot{B}} &= \frac1{3!} \left( \Gamma^{I_1}_{A\dot{A}} \Gamma^{I_2}_{B\dot{A}}\Gamma^{I_3}_{B\dot{B}} \pm\text{$5$ permutations}\right) \,.
\end{align}
For more details on the use of these gamma matrices see~\cite{Nicolai:2004nv,Kleinschmidt:2006dy}.

The explicit commutations relations of~\eqref{e8} in the basis~\eqref{e8bas} are then given by
\begin{subequations}
\label{e8rels}
\begin{align}
\lb X^{IJ} , X^{KL} \rb &= \delta^{JK} X^{IL} - \delta^{IK} X^{JL} - \delta^{JL} X^{IK} + \delta^{IL} X^{JK}\,,\\
\lb X^{IJ} , Y^A \rb &= -\frac12  \Gamma^{IJ}_{AB} Y^B\,,\\
\lb Y^A, Y^B \rb &= \frac14 \Gamma^{IJ}_{AB} X^{IJ}\,.
\end{align}
\end{subequations}
The first line is just the $\khor\cong \mf{so}(16)$ Lie algebra while the others are manifestations of the Cartan decomposition's properties $\big[\khor,\phor\big]\subset \phor$ and $\big[\phor,\phor\big]= \khor$. In particular, the last relation can be viewed as a projection from $\text{Alt}^2(\phor)$ to the factor that is isomorphic to the adjoint module of $\khor$. 

The affine algebra is then obtained according to~\eqref{eq:aff} by introducing the loop generators
\begin{align}
\label{loopgens}
X_m^{IJ} = t^m \otimes X^{IJ}\,,\quad Y^A_m = t^m\otimes Y^A
\end{align}
for $m\in \mathbb{Z}$. The commutation relations, including the central element $K$ and derivation $d$, are then given by 
\begin{align}
\lb X^{IJ}_m , X^{KL}_n \rb &= \delta^{JK} X^{IL}_{m+n} - \delta^{IK} X^{JL}_{m+n} - \delta^{JL} X^{IK}_{m+n} + \delta^{IL} X^{JK}_{m+n}\nn\\
&\hspace{5mm}  -m \left(\delta^{IK}\delta^{JL}- \delta^{JK} \delta^{IL}\right) \delta_{m,-n} K\,,\\[2mm]
\lb X^{IJ}_m , Y^A_n \rb &= -\frac12  \Gamma^{IJ}_{AB} Y^B_{m+n}\,,\\[2mm]
\lb Y^A_m, Y^B_n \rb &= \frac14 \Gamma^{IJ}_{AB} X^{IJ}_{m+n} + m \delta^{AB} \delta_{m,-n} K\,,
\\[2mm]
\lb d , X^{IJ}_m \rb &= - m X^{IJ}_m\;\; , \quad
\lb d , Y^A_m \rb = - m Y^A_m\,,
\end{align}
and $K$ commutes with everything.

As is well known, there is an action of the Virasoro algebra on loop 
algebras~\cite{Kac1990,Goddard:1986bp}. The Virasoro algebra is a 
central extension of the Witt algebra generated by the operators
\begin{align}
L_m = -t^{m+1} \frac{d}{dt} \quad(m\in\mathbb{Z})
\end{align}
acting on Laurent polynomials and the non-trivial commutators with the $\mf{e}_9$ loop generators are
\begin{align}
\lb L_m, X^{IJ}_n \rb = - n X^{IJ}_{m+n} \,, \quad \lb L_m, Y^A_n \rb = - nY^A_{m+n} \,.
\end{align}
$L_0$ therefore acts like $d$. The commutators among the Virasoro generators is
\begin{align}
\lb L_m, L_n \rb = (m-n) L_{m+n} + \frac{c_{\text{Vir}}}{12} m (m^2-1) \delta_{m,-n} K
\end{align}
with the same central element $K$ as in the affine algebra and $c_{\text{Vir}}$ a free coefficient. 
In any given highest or lowest weight representation of the affine algebra, one can find a realisation of the Virasoro algebra in the universal enveloping algebra of the loop algebra via the Sugawara construction and this fixes $c_{\text{Vir}}$, see~\cite{Goddard:1986bp} for a review.

\subsection{\texorpdfstring{Involutory subalgebra $\mf{k}(\mf{e}_9)$}{Involutory subalgebra k(e9)}}

The involutory subalgebra $\mf{k}\equiv \mf{k}(\mf{e}_9)$ was defined in~\eqref{eq:kdef}. We shall give explicit forms of the filtered structure and the parabolic model discussed in sections~\ref{sec:aff} and~\ref{sec:par}, respectively.

\subsubsection{Filtered basis}

We recall that the central element $K$ and the derivation $d$ are not invariant under the involution $\omega$ and therefore not part of $\mf{k}(\mf{e}_9)$.
The generators of $\mf{k}(\mf{e}_9)$ can be expressed in terms of~\eqref{loopgens} according to
\begin{subequations}
\label{ke9bas1}
\begin{align}
\cX^{IJ}_{m} &= \frac12 (t^m +t^{-m})\otimes X^{IJ} \equiv \frac12 \left(X_m^{IJ} + X_{-m}^{IJ}\right)\,,\\
\cY^A_{m} &= \frac12 (t^m - t^{-m})  \otimes Y^A \equiv\frac12\left( Y^A_m - Y^A_{-m}\right) \,.
\end{align}
\end{subequations}
where $t$ is the spectral parameter and $m\in \mathbb{N}$. Note that $\cY^A_0\equiv 0$. Here, we have introduced a factor of $1/2$ such that $\cX_0^{IJ} = X_0^{IJ}$ is also an $\mf{so}(16)$ Lie algebra with the same normalisation.

The complete $\mf{k}(\mf{e}_9)$ Lie algebra in the basis~\eqref{ke9bas1} reads
\begin{subequations}
\label{KE9rels1}
\begin{align}
\label{eq:XX16}
\lb \cX_{m}^{IJ} ,\cX_{n}^{KL}\rb &= \frac12\delta^{JK} \left(\cX_{m+n}^{IL} + \cX_{|m-n|}^{IL} \right) - \frac12\delta^{IK} \left(\cX_{m+n}^{JL} + \cX_{|m-n|}^{JL} \right) \nn\\
&\hspace{10mm}-  \frac12\delta^{JL} \left(\cX_{m+n}^{IK} + \cX_{|m-n|}^{IK} \right) +  \frac12\delta^{IL} \left(\cX_{m+n}^{JK} + \cX_{|m-n|}^{JK} \right)
\,,\\[2mm]
\label{eq:XY16}
\lb \cX_{m}^{IJ}, \cY_{n}^A \rb &= -\frac14 \Gamma^{IJ}_{AB} \left( \cY_{m+n}^B - \text{sgn}(m-n) \cY_{|m-n|}^B\right)\,,\\[2mm]
\label{eq:YY16}
\lb \cY_{m}^A, \cY_{n}^B \rb &= \frac18 \Gamma^{IJ}_{AB} \left( \cX_{m+n}^{IJ} - \cX_{|m-n|}^{IJ}\right)\,,
\end{align}
\end{subequations}
where $\text{sgn}(k)$ is the sign function with $\text{sgn}(0)=0$.

In this $\mf{so}(16)$-covariant formulation, $\mf{k}(\mf{e}_9)$ can be described as the algebra generated by the $\cX_0^{IJ}$, $\cX^{IJ}_1$ and $\cY^A_1$ with the property that the $\cX^{IJ}_0$ form an $\mf{so}(16)$ algebra and that $\cX^{IJ}_1$ and $\cY^A_1$ transform correctly under this algebra according to~\eqref{KE9rels1}. The $\cX_0^{IJ}$, $\cX^{IJ}_1$ and $\cY^A_1$ obey
\begin{subequations}
\label{eq:ber1}
\begin{align}
7\, \Gamma^{IJ}_{AB}  \lb \cY_{1}^A , \cY_{1}^B \rb  -32 \lb \cX_{1}^{IK} ,\cX_{1}^{KJ}\rb  &=-448 \cX_{0}^{IJ}\,,\\
\Gamma^{I_1\ldots I_6}_{AB} \lb \cY_{1}^A , \cY_{1}^B \rb &=0\,.
\end{align}
\end{subequations}
These are the $\mf{so}(16)$-covariant forms of the Berman relation $\lb x_0, \lb x_0,x_1\rb\rb=-x_1$ in a Chevalley--Serre basis 
\begin{align}
x_i\coloneqq e_i-f_i
\end{align}
and where we use the convention that $0$ is the affine node of the 
$\mf{e}_9$ Dynkin diagram that attaches to the adjoint node, labelled $1$, of the $\mf{e}_8$ Dynkin diagram, see figure~\ref{fig:e9dynk}.
To write out the `affine' Berman generator $x_0$ explicitly in terms of the basis
(\ref{ke9bas1}), we need to make use of the SO(8) decompositions (A.9) in Appendix A 
of \cite{Kleinschmidt:2006dy}; using the notation and transformations of Appendix B of that reference there we have
\begin{equation}
x_0 =-\frac12 \gamma^1_{\alpha\dot\beta} \big( \cX_1^{\alpha\dot\beta} 
                + \cY_1^{\alpha\dot\beta} \big)
\end{equation}
with the SO(8) gamma matrices $\gamma^i_{\alpha\dot\beta}$ (for $i=1,...,8$). The remaining
 Berman generators in this basis are given by
 \begin{align}
 x_i =  \frac14\left(\gamma^{i,i+1}_{\alpha\beta} \cX_0^{\alpha\beta} +
     \gamma^{i,i+1}_{\dot\alpha\dot\beta} \cX_0^{\dot\alpha\dot\beta} \right) \;\;\; \mbox{(for $i=1,...,7$)} \;\;, 
\quad  x_8 = -\frac12 \gamma^{678}_{\alpha\dot\beta} \cX_0^{\alpha\dot\beta} \,.
\end{align}

The Virasoro algebra likewise can be restricted to an involutory subalgebra~\cite{Julia}.
The `maximal compact' subalgebra of the Virasoro algebra is generated by 
\begin{equation}
\label{comVir}
\cK_m \coloneqq L_m - L_{-m} \,\equiv  \, - \big( t^{m+1} - t^{-m+1} \big) \frac{d}{dt}   \qquad \mbox{for $m\geq 1$.}
\end{equation}
The relevant commutation relations when acting on $K(\mf{e}_9)$ in the basis~\eqref{ke9bas1} read
\begin{align}
\lb \cK_m\,,\, \cX^{IJ}_{n} \rb  & =  - n \big( \cX^{IJ}_{m+n}  - \cX^{IJ}_{|m-n|} \big)    \nn\\[2mm]
\lb \cK_m\,,\, \cY^{A}_{n} \rb  & =  - n \big( \cY^{A}_{m+n}  - \text{sgn}(m-n)   \cY^{A}_{|m-n|} \big)
\end{align}
and
\begin{equation}
\lb \cK_m \,,\, \cK_n \rb \,=\, (m-n) \cK_{m+n}  - {\rm sign}(m-n) (m+n) \cK_{|m-n|}
\end{equation}
Note that the central term drops out here as well.

\subsubsection{Parabolic model}

As discussed in sections~\ref{sec:par} and~\ref{sec:verm}, it is very useful for constructing representations of $\mf{k}(\mf{e}_9)$ to consider the parabolic algebras $\mf{N}\left(\mathbb{K}[[u]]\right)$ and $\mf{N}\left(\mathbb{K}[u]\right)$ defined in~\eqref{eq:Ndef} and~\eqref{eq:Ndec}, respectively. Since polynomials have a graded product,  $\mf{N}\left(\mathbb{K}[u]\right)$ is a graded Lie algebra. We write the variable of the polynomial as $u$ as in section~\ref{sec:par} and it is related to the variable $t$ in the filtered basis by $u=\frac{1-t}{1+t}$ so that expansions around $u=0$ are expansions around $t=1$ and vice versa. 

The definition of the basis generators of $\mf{N}\left(\mathbb{K}[u]\right)$ is then explicitly
\begin{align}
\label{ke9bas2}
\tA_{2m}^{IJ} \coloneqq u^{2m} \otimes X^{IJ} \,,
\quad\quad\quad
\tS_{2m+1}^A \coloneqq u^{2m+1} \otimes Y^A
\quad\quad \text{for $m\in\mathbb{N}$}\,.
\end{align}

The Lie algebra of these generators is graded and given by 
\begin{subequations}
\label{ke9rels2}
\begin{align}
\lb \tA_{2m}^{IJ} , \tA_{2n}^{KL} \rb &= \frac12 \delta^{JK} \tA_{2(m+n)}^{IL}- \frac12 \delta^{IK} \tA_{2(m+n)}^{JL}- \frac12 \delta^{JL} \tA_{2(m+n)}^{IK}+ \frac12 \delta^{IL} \tA_{2(m+n)}^{JK}
\,,\\[2mm]
\lb \tA_{2m}^{IJ} , \tS_{2n+1}^{KL} \rb &= -\frac14\Gamma^{IJ}_{AB} \tS_{2(m+n)+1}^{B}\,,\\[2mm]
\label{ke9rels2c}
\lb \tS_{2m+1}^A, \tS_{2n+1}^B \rb &= \frac18 \Gamma^{IJ}_{AB} \tA_{2(m+n+1)}^{IJ}\,.
\end{align}
\end{subequations}
The generators of $\mf{N}\left(\mathbb{K}[[u]]\right)$ also include infinite linear combinations of~\eqref{ke9bas2} since $\mf{N}\left(\mathbb{K}[[u]]\right)$ is constructed using power series rather than polynomials.

The maps (\ref{eq:ansatz for k}) and (\ref{eq:ansatz for p}) from the filtered to the
parabolic bases now read
\begin{equation}
\rho_\pm (\cX_n^{IJ} ) = (\pm 1)^n \frac12\sum_{k\geq 0} a^{(n)}_{2k} \tA_{2k}^{IJ} \quad , \qquad
\rho_\pm (\cY_n^A) = (\pm 1)^n \frac12 \sum_{k\geq 0} a^{(n)}_{2k+1} \tS_{2k+1}^A\,,
\end{equation}
where the factors of $\tfrac12$ are due to our definition~\eqref{ke9bas1}.
More specifically,
the images of the first few generators of $\mf{k}(\mf{e}_9)$ according to proposition~\ref{prop:ansatz for hom} are given in the above basis by 
\begin{align}
\label{par2tri}
\rho_+ \left( \cX_{0}^{IJ}\right) &=  \,\tA_{0}^{IJ}\,,\nn\\
\rho_+ \left( \cX_{1}^{IJ}\right) &=  \, \tA_{0}^{IJ} + 2 \sum_{k\geq 1} \tA_{2k}^{IJ}\,,\nn\\
\rho_+ \left( \cX_{2}^{IJ}\right) &=  \, \tA_{0}^{IJ} + 8 \sum_{k\geq 1} k \, \tA_{2k}^{IJ}\,,\nn\\
\rho_+ \left( \cX_{3}^{IJ}\right) &=  \, \tA_{0}^{IJ} + 2 \sum_{k\geq 1} (1+8 k^2)  \, \tA_{2k}^{IJ}\, \nn\\
\rho_+ \left( \cX_{4}^{IJ}\right) &=  \, \tA_{0}^{IJ} + \frac{32}{3} \sum_{k\geq 1} (k+2 k^3)  \, \tA_{2k}^{IJ}\, \nn\\
\rho_+ \left( \cX_{5}^{IJ}\right) &=  \, \tA_{0}^{IJ} + \frac{2}{3} \sum_{k\geq 1} (3+40k^2+32 k^4)  \, \tA_{2k}^{IJ}
\end{align}
for the $\cX^{IJ}_m$. For the $\cY^A_m$ the relations read
\begin{align}
\rho_+\left( \cY_1^A \right)&=  - 2 \sum_{k\geq 0} \tS^A_{2k+1}\,,\nn\\
\rho_+\left( \cY_2^A \right)&=  - 4 \sum_{k\geq 0} (2k +1) \tS_{2k+1}^A\,,\nn\\
\rho_+\left( \cY_3^A \right)&=  - 2 \sum_{k\geq 0} (8k^2+8k+3) \tS^A_{2k+1}\,,\nn\\
\rho_+\left( \cY_4^A \right) &=  - \frac{8}3 \sum_{k\geq 0} (3+10k+12k^2+8k^3) \tS_{2k+1}^A\,,\nn\\
\rho_+\left( \cY_5^A \right) &=  - \frac23 \sum_{k\geq 0} (15+56k+88k^2+64k^3+32k^4) \tS_{2k+1}^A\,.
\end{align}

We note that, unlike~\eqref{eq:ber1}, there is no known presentation of $\mf{N}\left(\mathbb{K}[[u]]\right)$ as a finitely generated algebra, say by $\tA_0^{IJ}$ and $\tS_1^A$, with a finite number of Berman-like relations.
Using the algebra~\eqref{ke9rels2} we find the following Berman-type relations for all $k\geq 1$
\begin{subequations}
\label{eq:ber162}
\begin{align}
7 \,\Gamma^{IJ}_{AB} \sum_{\substack{k_1\,,\,k_2\geq 0\\k_1+k_2=k-1}} \lb \tS_{2k_1+1}^A \,,\, \tS_{2k_2+1}^B \rb -32\sum_{\substack{k_1,k_2\geq 1\\k_1+k_2=k}} \lb \tA_{2k_1}^{IK} , \tA_{2k_2}^{KJ} \rb &=
448
\tA_{2k}^{IJ} \,,\\
\Gamma^{I_1\ldots I_6}_{AB} \sum_{\substack{k_1\,,\,k_2\geq 0\\k_1+k_2=k-1}} \lb \tS_{2k_1+1}^A \,,\, \tS_{2k_2+1}^B \rb &=0\,,
\end{align}
\end{subequations}
where we have evaluated the commutators involving $\tA_0$ in the first line.

The maximal compact Virasoro generators $\cK_m$ introduced in~\eqref{comVir} can also be expressed using the variable $u$. In particular,
\begin{equation}
\cK_1 = L_1 - L_{-1} = \big( 1 - t^2) \frac{d}{dt}  = 2u \frac{d}{du}
\end{equation}
generates an SO(1,1) group, and acts as a counting operator on the basis of the
parabolic model~\eqref{ke9bas2}:
\begin{equation}
\lb \cK_1 \,,\, \tA_{2k}^{IJ} \rb \,=\, 4k \,\tA_{2k}^{IJ} \;\; , \quad
\lb \cK_1 \,,\, \tS_{2k+1}^A \rb \,=\,  2(2k+1) \, \tS^A_{2k+1}
\end{equation}

More generally the operators $\cK_m$ admit the following realisation as differential
operators
\begin{equation}
\cK_m \,=\, \frac12 (\pm 1)^m \frac1{(1 - u^2)^{m-1}} \big[ (1 - u)^{2m} - (1 +u)^{2m} \big] \frac{d}{du}
\end{equation}
Proceeding as before we find for instance
\begin{equation}
[ \cK_2, \tA_{2m}^{IJ} ] \,=\, 8m \big( \tA_{2m}^{IJ} + 2\tA_{2m+2}^{IJ}+ 2\tA_{2m+4}^{IJ} + \cdots \big)
\end{equation}
so for $m\geq 2$ these operators mix all levels.

\subsection{Representations}

Representations of $\mf{k}(\mf{e}_9)$ can be constructed via the technique described in section~\ref{sec:verm}. This means that we construct a basis of the universal enveloping algebra  of $\mf{N}_{+}\subset\mf{N}$ (given by all generators in~\eqref{ke9bas2} with degree greater than 0) and let this act on a $\khor\cong \mf{so}(16)$ representation $\tV_0$ as the initial vector space. We shall exploit that everything is $\mf{so}(16)$ covariant and graded.

\subsubsection{Universal enveloping algebra basis}

The basis~\eqref{eq:vermbas} of $\mathcal{U}\left(\mf{N}_{+}\right)$ at the first few levels becomes 
\begin{subequations}
\begin{align}
\ell=0: \quad&1\\
\ell=1:\quad & \tS_1^A\\
\ell=2: \quad& \tS_1^{(A} \tS_1^{B)} ,\quad  \tA_2^{\alpha}\\
\ell=3: \quad& \tS_1^{(A} \tS_1^B \tS_1^{C)},\quad  \tS_1^A \tA_2^{\alpha} ,\quad \tS_3^A\\
\ell=4: \quad& \tS_1^{(A} \tS_1^B \tS_1^C \tS_1^{D)},\quad  \tS_1^{(A} \tS_1^{B)} \tA_2^{\alpha}  ,\quad  \tS_1^B \tS_3^A,\quad \tA_2^{(\alpha} \tA_2^{\beta)},\quad \tA_4^{\alpha}
\end{align}
\end{subequations}
where we now use $\alpha\equiv [IJ]$ for $I<J$ to denote an adjoint index of $\mf{so}(16)$.
The symmetrisations $(\cdots)$ for similar generators on the same level are necessary to 
implement the ordering in accordance with the PBW theorem. 

In terms of $\mf{so}(16)$ representations the first levels $\mathcal{U}_\ell$ of $\mathcal{U}(\mf{N}_{+})$ are\footnote{The parentheses are used to group the representations according to the different words in the induced representation module list after decomposing into $\mf{so}(16)$.}
\begin{subequations}
\begin{align}
\ell =0:\quad& {\bf 1}\\
\ell=1: \quad & {\bf 128}_s\\
\ell=2: \quad & \Big( {\bf 1}\oplus {\bf 1820}\oplus {\bf 6435}_+\Big) \oplus {\bf 120}\\
\ell=3: \quad & \Big({\bf 128}_s \oplus {\bf 13312}_s \oplus {\bf 161280}_s \oplus {\bf 183040}_s\Big) \nn\\
&\quad  \oplus\Big( {\bf 128}_s \oplus {\bf 1920}_c \oplus {\bf 13312}_s\Big) \oplus {\bf 128}_s
\end{align}
\end{subequations}
Here, we have labelled the $\mf{so}(16)$ representations by their real dimensions. A translation to highest weight labels can be found in appendix~\ref{app:so16reps}.

The module $\tV$ as in (\ref{eq:induced_V}) built from an $\khor\cong\mf{so}(16)$ representation $\tV_0$ is graded as
\begin{align*}
\tV = \bigoplus_{\ell=0}^\infty \tV_\ell\,,
\hspace{20mm} \tV_\ell \coloneqq \mathcal{U}_\ell \otimes \tV_0\,
\end{align*}
and each level decomposes further as a $\mf{so}(16)$-representation. For instance, in the case of the $\tV_0={\bf 16}$ we obtain
\begin{subequations}
\label{eq:V16}
\begin{align}
\tV_0 &= {\bf 16}\,,\\
\tV_1 &= {\bf 128}_c \oplus {\bf 1920}_s \,,\\
\label{eq:16V2}
\tV_2 &= \Big( {\bf 16} \oplus  {\bf 560} \oplus {\bf 4368} \oplus {\bf 11440} \oplus {\bf 24192} \oplus{\bf 91520}_+\Big)\oplus \Big( {\bf 16}  \oplus {\bf 560} \oplus {\bf 1344}\Big)\,,\\
\tV_3 &= \Big( {\bf 128}_c \oplus 2{\times} {\bf 1920}_s \oplus {\bf 13312}_c \oplus 2{\times} {\bf 56320}_s \oplus {\bf 141440}_s \oplus {\bf 161280}_c \nn\\
&\hspace{10mm}\oplus {\bf 326144}_s \oplus {\bf 439296}_c \oplus {\bf 2036736}_s \oplus {\bf 2489344}_s\Big)\nn\\
&\quad \oplus
\Big( 2{\times}{\bf 128}_c \oplus 3{\times}{\bf 1920}_s \oplus {\bf 13312}_c \oplus {\bf 15360}_c \oplus {\bf 56320}_s \oplus {\bf 141440}_s\Big)\nn\\
&\quad \oplus \Big({\bf 128}_c \oplus {\bf 1920}_s\Big)\,.
\end{align}
\end{subequations}

\subsubsection{A quotient example}

As is clear from~\eqref{eq:V16}, the module grows very rapidly and it is desirable to find smaller $\mf{k}(\mf{e}_9$) representations by identifying invariant subspaces and taking quotients.

The simplest quotient example is to consider $\tW_{s=1/2}=\tV^{(1)}$ in the notation~\eqref{eq:quotN} to be given by all spaces of degree greater than zero, then we obtain as $\mf{k}(\mf{e}_9)$ representation simply $\tV / \tW_{s=1/2}\cong \tV_0 \cong {\bf 16}$ which is nothing but the irreducible spin-1/2 representation appearing in supergravity~\cite{Nicolai:2004nv}.

A non-trivial example can be obtained by looking at the construction~\eqref{quot2} that uses the invariant subspaces $\tW$ generated by an $\mf{so}(16)$ representation $\tW_k$ sitting at a given level. As the example we take $\tW_1={\bf 1920}_s$ within~\eqref{eq:V16}.  In order to describe this, we shall need a more explicit parametrisation of the elements of the module's homogeneous parts $\tV_{\ell}$ for $0\leq \ell\leq 2$. At level $\ell=0$ we need an element of the  $16$-dimensional defining representation of $\mf{so}(16)$ that we write as $\varphi_0^I$.

The elements of $\tV_1\cong{\bf 128}_s\otimes {\bf 16}$ are of the form $\tS_1^A \varphi_0^I$ and decompose into a conjugate spinor ${\bf 128}_c$ and a traceless vector-spinor ${\bf 1920}_s$
\begin{align*}
\tS_1^A \varphi^I = \Gamma^{I}_{A\dot{A}} \chi_1^{\dot{A}} + \chi^{IA}_1
\end{align*}
according to $\tV_1 = {\bf 128}_c \oplus {\bf 1920}_s$. Note that the occurrence of ${\bf 128}_c$ in the above tensor product is tied to the existence of a suitable $\Gamma$-matrix $\Gamma^{I}_{A\dot{A}}$. The condition that projects onto the ${\bf 1920}_s$  is
\begin{align*}
 \chi^{IA}_1 = \tS_1^A \varphi_0^I - \frac1{16} (\Gamma^I\Gamma^J)_{AB} \tS_1^B \varphi_0^J
\end{align*}
and the quotient we want to consider is the one where this component and all vectors generated from it are set to zero. In other words, all elements of the quotient module $\tV/\tW$ satisfy
\begin{align}
\label{1920rel}
\tS_1^B \varphi_0^I = \frac1{16} (\Gamma^I\Gamma^J)_{BC} \tS_1^C \varphi_0^J
\end{align}
and any relations obtained from it by acting with $\mf{k}(\mf{e}_9)$.

To find what this imposes on $\tV_2$ as given in~\eqref{eq:16V2} we parametrise all its elements as
\begin{align}
\label{lev2}
\tS_1^A \tS_1^B \varphi_0^I = \delta^{AB} \varphi^I_2 + \Gamma^{J_1J_2}_{AB} \varphi^{J_1J_2;I}_2 + \Gamma^{J_1\ldots J_4}_{AB} \varphi^{J_1\ldots J_4;I}_2 +\Gamma^{J_1\ldots J_8}_{AB} \varphi^{J_1\ldots J_8;I}_2 \,,
\end{align}
This formula follows from the $\mf{so}(16)$ tensor product ${\bf 128}_s\otimes {\bf 128}_s$ that is relevant for the word $\tS_1^A \tS_1^B $ and that we decompose into its symmetric and anti-symmetric parts
\begin{align*}
\mathrm{Sym}^2 \left({\bf 128}_s\right) &= {\bf 1} \oplus {\bf 1820} \oplus {\bf 6435}_+\,,\nn\\
\mathrm{Alt}^2 \left({\bf 128}_s\right) &= {\bf 120} \oplus {\bf 8008}\,.
\end{align*}
The first line consists of a scalar, a four-form and a self-dual eight-form of $\mf{so}(16)$, while the second line represents a two-form and a six-form. These intertwiners from ${\bf 128}_s\otimes {\bf 128}_s$ to $p$-forms are given by the $\Gamma$-matrices $\Gamma^{I_1\ldots I_p}_{AB}$. As the anti-symmetric product $\tS_1^{[A} \tS_1^{B]}$ is proportional to the commutator~\eqref{ke9rels2c} that does not contain a six-form, the ansatz~\eqref{lev2} does not contain a term in $\Gamma^{I_1\ldots I_6}_{AB}$. The final $\mf{so}(16)$ tensor product on the left-hand side of~\eqref{lev2} is then to multiply ${\bf 128}_s\otimes {\bf 128}_s$ (written as $p$-forms) with ${\bf 16}$ which is the representation of $\varphi_0^I$. These tensor products are written using a semi-colon, so that for instance

\begin{align*}
\varphi_2^{J_1J_2; I} \in {\bf 120} \otimes {\bf 16} = {\bf 16} \oplus {\bf 560} \oplus {\bf 1344}\,.
\end{align*}

Acting on~\eqref{1920rel} with $\tS_1^A$ then leads to the relation
\begin{align}
&\quad \delta_{AB} \varphi^I_2 + \Gamma^{J_1J_2}_{AB} \varphi^{J_1J_2;I}_2 + \Gamma^{J_1\ldots J_4}_{AB} \varphi^{J_1\ldots J_4;I}_2 +\Gamma^{J_1\ldots J_8}_{AB} \varphi^{J_1\ldots J_8;I}_2\nn\\
&\overset{!}{=}\frac1{16} (\Gamma^I\Gamma^K)_{BC}\bigg(\delta^{AC} \varphi^K_2 - \Gamma^{J_1J_2}_{CA} \varphi^{J_1J_2;K}_2 + \Gamma^{J_1\ldots J_4}_{CA} \varphi^{J_1\ldots J_4;K}_2 +\Gamma^{J_1\ldots J_8}_{CA} \varphi^{J_1\ldots J_8;K}_2\bigg)\nn\\
&= \frac1{16} \delta_{AB} \varphi_2^I - \frac{1}{16} \Gamma^{IJ}_{AB} \varphi_2^J + \frac{1}{16} \Gamma^{J_1J_2}_{AB} \varphi_2^{J_1J_2;I} - \frac1{16} \Gamma^{IJ_1J_2J_3}_{AB} \varphi_2^{J_1J_2;J_3}+\frac18 \Gamma^{J_1J_2}_{AB} \varphi_2^{IJ_1;J_2} - \frac18 \Gamma^{IJ}_{AB} \varphi_2^{JK;K} \nn\\
&\quad  + \frac18 \delta_{AB} \varphi_2^{IJ;J} + \frac1{16} \Gamma^{J_1\ldots J_4}_{AB} \varphi^{J_1\ldots J_4; I} - \frac1{16} \Gamma^{IJ_1\ldots J_5}_{AB} \varphi_2^{J_1\ldots J_4;J_5} - \frac14 \Gamma^{IJ_1\ldots J_3}_{AB} \varphi^{J_1\ldots J_3K;K}_2\nn\\
&\quad + \frac14 \Gamma^{J_1\ldots J_4}_{AB} \varphi^{I J_1\ldots J_3;J_4}_2 + \frac{3}{4} \Gamma^{J_1J_2}_{AB} \varphi^{IJ_1J_2K;K}_2 + \frac1{16} \Gamma^{J_1\ldots J_8}_{AB} \varphi^{J_1\ldots J_8;I}_2 - \frac1{16}\Gamma^{IJ_1\ldots J_9}_{AB} \varphi_2^{J_1\ldots J_8;J_9} \nn\\
&\quad -\frac12\Gamma^{IJ_1\ldots J_7}_{AB} \varphi_2^{J_1\ldots J_7K;K} -\frac12 \Gamma^{J_1\ldots J_8}_{AB} \varphi_2^{I J_1\ldots J_7;J_8} + \frac72 \Gamma^{J_1\ldots J_6}_{AB} \varphi_2^{IJ_1\ldots J_6K;K}\,.
\end{align}
Projecting this onto the various irreducible pieces in~\eqref{eq:16V2} leads to the conditions 
\begin{align}
\varphi_2^{IJ;J} &= \frac{15}{2} \varphi_2^I &&\text{(relation between the two ${\bf 16}$)}\nn\\
\varphi_2^{I_1I_2I_3J;J} &= \frac{13}{12} \varphi_2^{[I_1I_2;I_3]} &&\text{(relation between the two ${\bf 560}$)}
\end{align}
and the fact that all other irreducible components must vanish. Therefore, at level $\ell=2$, the quotient is given by only
\begin{align}
\label{eq:quotlev2}
\tV_2/\tW_2 \cong {\bf 16} \oplus {\bf 560}\,,
\end{align}
a comparatively small subspace of~\eqref{eq:16V2}. Already at the next level the above computation becomes almost unfeasible. If one formally substracts the $\mf{so}(16)$-decompositions of $\mathcal{U}_{\ell+1}\otimes\tV_0$ and $\mathcal{U}_{\ell}\otimes\tW_1$ the result indicates that only ${\bf 128}_c$ survives at level three, and that there are only two ${\bf 16}$s at 
level four, after which the procedure terminates. In summary, the above computation shows
\begin{align}
\tV_0/\tW_0 \cong \tV_0 &\cong {\bf 16}\,,\quad\nn\\
\tV_1/\tW_1 &\cong {\bf 128}_c\,,\quad \nn\\
\tV_2/\tW_2 &\cong {\bf 16}\oplus {\bf 560}\,,\quad 
\end{align}
and we conjecture
\begin{align}
\tV_3/\tW_3 &\cong {\bf 128}_c \,,\quad \nn\\
\tV_4/\tW_4 &\cong 2\times {\bf 16}\,,\quad \nn\\
\tV_\ell/\tW_\ell &\cong 0\ \forall\,\ell\geq5\,.
\end{align}
This is related to the analogue of the spin-$\frac52$ representation studied in~\cite{Kleinschmidt:2016ivr}. The spin-3/2 representation of supergravity~\cite{Nicolai:2004nv,Kleinschmidt:2006dy} can also be obtained from this construction by taking a further quotient. More precisely, one quotients by all $\tV_\ell$ with $\ell>2$ and also by the ${\bf 560}$ representation in~\eqref{eq:quotlev2}. The remaining $\mf{so}(16)$ representations are ${\bf 16}\oplus{\bf 128}\oplus {\bf 16}$ that form one chiral half of the supergravity spin-3/2 fields.

\begin{rem}
As is evident from the analysis above, determining the quotient $\tV/\tW$ can become intricate quickly since the precise structure of the submodule $\tW$ is hard to analyse. In the case of complex simple Lie algebras a similar problem arises when constructing irreducible highest weight representations as quotients of Verma modules by the maximal proper submodule. In that case, there is a description of the quotient in terms of the Weyl character formula. A similar technique for representations of $\mf{k}$ is not known to the best of our knowledge.
\end{rem}

\newpage
\appendix


\section{\texorpdfstring{$\mf{k}\left(\mf{e}_9\right)\left(\mathbb{C}\right)$ as the quotient of a GIM-algebra}{k(e9) as the quotient of a GIM-algebra}}
\label{app:gim}

Generalised intersection algebras, GIM-algebras in short for Generalised
Intersection Matrix, are constructed similarly to Kac--Moody algebras.
One starts from a so-called \textit{generalised intersection matrix}
$A$ where one replaces the condition $A_{ij}\leq0$ for $i\neq j$
by $A_{ij}\leq0\Leftrightarrow A_{ji}\leq0$ and $A_{ij}>0\Leftrightarrow A_{ji}>0$.
GIMs can be visualised quite neatly by drawing solid lines for $A_{ij}<0$
and dotted edges for $A_{ij}>0$. In \cite{Berman:1989}, Berman explores
a connection between his involutory subalgebras and GIM-algebras.
As it turns out due to the work of Slodowy, GIM-algebras fall into
two classes. The first class consists of those GIM-algebras which
are in fact isomorphic to a Kac--Moody-algebra and the second class
are those which are isomorphic to an involutory subalgebra of a Kac--Moody-algebra
according to Berman's construction. This isomorphism relies on doubling
the number of vertices of the respective GIM-diagram and the involution
that is used involves a diagram automorphism of the new diagram which
is of Kac--Moody-type. So even though our $\mathfrak{k}$ is an involutory
subalgebra it is not of the type that is directly isomorphic to a
GIM-algebra. For $\mathfrak{k}\left(E_{9}\right)\left(\mathbb{C}\right)$
it turns out that it is a quotient of a GIM-algebra, although we do not know the precise structure of the defining ideal. In particular we do not know whether it is generated by (\ref{eq:non-defining relation 1})--(\ref{eq:non-defining relation 5}) or whether one needs additional relations. We will collect
the most essential definitions and results here (see the work of Slodowy
\cite{Slodowy} for more details) and state our result in proposition \ref{prop:gim-quotient}.
\begin{defn}
Let $I$ be a finite index set and $\mathfrak{h}$ a $\mathbb{C}$-vector
space of dimension $r$. Let $\Delta^{\vee}\coloneqq\left\{ h_{i}\ \vert\ i\in I\right\} \subset\mathfrak{h}$
and $\Delta\coloneqq\left\{ \alpha_{i}\ \vert\ i\in I\right\} \subset\mathfrak{h}^{\ast}$.
Then $\left(\mathfrak{h},\Delta^{\vee},\Delta\right)$ is called a
$\mathbb{C}$-root basis, its reductive rank is defined to be equal
to $r$, whereas its semi-simple rank is defined to be $\vert I\vert$.
Associate a matrix $A$ to $\left(\mathfrak{h},\Delta^{\vee},\Delta\right)$,
called its structure matrix, by setting 
\begin{equation}
A_{ij}\coloneqq\alpha_{j}\left(h_{i}\right)\ \forall\,i,j\in I.\label{eq:structure matrix}
\end{equation}
A root basis is called free if both $\Delta^{\vee}$ and $\Delta$
are linearly independent.
\end{defn}
\begin{defn}
Let $A\in\mathbb{Z}^{\ell\times \ell}$ such that 
\begin{eqnarray*}
(i)\ \,\quad A_{ii} & = & 2\ \forall\,i=1,\dots,\ell\\
(ii)\quad A_{ij} & < & 0\ \Leftrightarrow\ A_{ji}<0\ \forall\,i\neq j\\
(iii)\quad A_{ij} & > & 0\ \Leftrightarrow\ A_{ji}>0\ \forall\,i\neq j,
\end{eqnarray*}
then $A$ is called a generalised intersection matrix (GIM). A GIM
$A$ is called symmetrisable if there exist $D,B\in\mathbb{Q}^{\ell\times \ell}$
such that $D$ is diagonal and $B$ is symmetric and it holds $A=DB$.
A root basis $\left(\mathfrak{h},\Delta^{\vee},\Delta\right)$ whose
structure matrix is a generalised intersection matrix is called a
GIM-root basis.
\end{defn}
\begin{defn}
Let $\left(\mathfrak{h},\Delta^{\vee},\Delta\right)$ be a GIM-root
basis with structure matrix $A$. Then $\underline{\Delta}(A)$ is
a coloured, weighted graph with vertices $\Delta$ with:
\begin{enumerate}
\item Two vertices $i$ and $j$ are connected by a dotted edge if $A_{ij}=\alpha_{j}\left(h_{i}\right)>0$.
\item Two vertices $i$ and $j$ are connected by a solid edge if $A_{ij}=\alpha_{j}\left(h_{i}\right)<0$.
\item There is no edge between two vertices $i$ and $j$ if $A_{ij}=0$.
\item The edges $(i,j$) are weighted by the weight $m_{ij}$ according
to the following table 
\begin{center}\begin{tabular}{|c|c|c|c|c|c|}
\hline 
$\alpha_{i}\left(h_{j}\right)\cdot\alpha_{j}\left(h_{i}\right)$ & $0$ & $1$ & $2$ & $3$ & $\geq4$\tabularnewline
\hline 
\hline 
$m_{ij}$ & $2$ & $3$ & $4$ & $6$ & $\infty$\tabularnewline
\hline 
\end{tabular}\end{center}
\end{enumerate}
The weights $m_{ij}=2,3$ are not set apart graphically in the
diagram's visualisation. In the case of $m_{ij}=4,6$ one draws an
arrow towards the node $j$ if $\left|A_{ji}\right|=\left|\alpha_{i}\left(h_{j}\right)\right|>\left|\alpha_{j}\left(h_{i}\right)\right|=\left|A_{ij}\right|$.
\end{defn}
\begin{defn}
\label{def:GIM-algebra}Let $\left(\mathfrak{h},\Delta^{\vee},\Delta\right)$
be a GIM-root basis with structure matrix $A$ and let $\mathfrak{f}$
be the free Lie algebra over $\mathbb{C}$ generated by $\mathfrak{h}$
and elements $e_{\alpha},e_{-\alpha}$ for $\alpha\in\Delta$. Let
$\mathfrak{I}$ be the ideal in $\mathfrak{f}$ generated by the relations
(identify $h_{-\alpha}\equiv-h_{\alpha}$ for $\alpha\in-\Delta$)
\begin{equation*}
\left[h,h'\right] =  0,\quad \left[h,e_{\alpha}\right]  =  \alpha(h)e_{\alpha}\qquad \forall\,h,h'\in\mathfrak{h},\alpha\in\pm\Delta,\\
\end{equation*}
\begin{equation*}
\left[e_{\alpha},e_{-\alpha}\right]  =  h_{\alpha},\quad 
\text{ad}\left(e_{\alpha}\right)^{\text{max}\left(1,1-\beta\left(h_{\alpha}\right)\right)}\left(e_{\beta}\right)=0\quad
\forall\,\alpha,\beta\in\pm\Delta \text{ and $\alpha\neq -\beta$}.
\end{equation*}

Set $\mathfrak{g}\coloneqq\mathfrak{f}\diagup\mathfrak{I}$ then $\mathfrak{g}$
is called the GIM-Lie algebra to $\left(\mathfrak{h},\Delta^{\vee},\Delta\right)$.
Note that every GIM $A$ has a free realisation $\left(\mathfrak{h},\Delta^{\vee},\Delta\right)$
that is unique up to isomorphism. In this sense one can associate
a GIM-Lie algebra $\mathfrak{gim}(A)$ to a structure matrix $A$
or equivalently a GIM-diagram $\underline{\Delta}(A)$.
\end{defn}

If one spells out the above relations for $A_{ij}<0$, one obtains
with $e_{i}=e_{\alpha_{i}}$, $f_{i}=e_{-\alpha_{i}}$ the familiar
Serre-relations
\[
\text{ad}\left(e_{i}\right)^{1-A_{ij}}\left(e_{j}\right)=0=\text{ad}\left(f_{i}\right)^{1-A_{ij}}\left(f_{j}\right),\quad\left[e_{i},f{}_{j}\right]=0=\left[e_{j},f_{i}\right]
\]
but for $A_{ij}>0$ one obtains 
\[
\text{ad}\left(e_{i}\right)^{1+A_{ij}}\left(f_{j}\right)=0=\text{ad}\left(f_{i}\right)^{1+A_{ij}}\left(e_{j}\right),\quad\left[e_{i},e{}_{j}\right]=0=\left[f_{i},f_{j}\right].
\]
One knows on abstract grounds that the complexification $\mathfrak{k}\left(A_{8}\right)\left(\mathbb{C}\right)$
of the canonical subalgebra $\mathfrak{k}\left(A_{8}\right)<\mathfrak{k}\left(E_{9}\right)$
is isomorphic to $B_{4}\left(\mathbb{C}\right)\cong\mathfrak{so}\left(9,\mathbb{C}\right)$.
Let us spell out the relationship between the description of $\mathfrak{k}\left(A_{8}\right)\left(\mathbb{C}\right)$
via Berman generators
\begin{align}
x_i = e_i - f_i
\end{align}
 and the usual description of $\mathfrak{g}\left(B_{4}\right)\left(\mathbb{C}\right)$
in terms of a Chevalley--Serre basis. Let $i_{1},i_{2},\dots,i_{k}\in\{1,\dots,9\}$
and set 
\begin{equation}
x_{\alpha_{i_{1}}+\dots+\alpha_{i_{k}}}\coloneqq\left[x_{i_{1}},\left[x_{i_{2}},\left[\dots,\left[x_{i_{k-1}},x_{i_{k}}\right]\right]\right]\right].\label{eq:nested ordered Bermans}
\end{equation}
Note that the order in the sum $\alpha_{i_{1}}+\dots+\alpha_{i_{k}}$
matters.  For $i<j$ define roots $\beta_{i,j}^{(1)},\dots,\beta_{i,j}^{(4)}\in\Delta\left(A_{8}\right)\subset\Delta\left(E_{9}\right)$
by 
\begin{equation}
\beta_{i,j}^{(1)}=\alpha_{2i}+\dots+\alpha_{2j-1}\ ,\ \beta_{i,j}^{(2)}=\alpha_{2i}+\dots+\alpha_{2j-2}\label{eq:def of =00005Cbeta_i,j part 1}
\end{equation}
\begin{equation}
\beta_{i,j}^{(3)}=\alpha_{2i-1}+\dots+\alpha_{2j-1}\ ,\ \beta_{i,j}^{(4)}=\alpha_{2i-1}+\dots+\alpha_{2j-2}.\label{eq:def of =00005Cbeta_i,j part 2}
\end{equation}
Now set 
\begin{equation}
e_{\varepsilon_{1}L_{i}+\varepsilon_{2}L_{j}}\coloneqq\frac{i}{2}\cdot\left(x_{\beta_{i,j}^{(1)}}-i\varepsilon_{2}x_{\beta_{i,j}^{(2)}}-i\varepsilon_{1}x_{\beta_{i,j}^{(3)}}-\varepsilon_{1}\varepsilon_{2}x_{\beta_{i,j}^{(4)}}\right)\  \forall\,i<j\in\{1,2,3,4\},\label{eq:root operators for so9}
\end{equation}
\begin{equation}
H_{j}\coloneqq-ix_{2j-1},\ \text{for }j=1,\dots,4\,,\label{eq:orthonormal Cartan}
\end{equation}
\begin{equation}
e_{\pm L_{j}}\coloneqq i\cdot\left(x_{\alpha_{2j}+\dots+\alpha_{8}}\mp ix_{\alpha_{2j-1}+\dots+\alpha_{8}}\right).\label{eq:small root operators for so9}
\end{equation}

\begin{prop}
Consider the abelian subalgebra $\mathfrak{h}_{B_{4}}\coloneqq\text{span}_{\mathbb{C}}\left\{ H_{1},\dots,H_{4}\right\} $
together with the linear functionals $L_{i}:\mathfrak{h}_{B_{4}}^{\ast}\rightarrow\mathbb{C}$
defined via $L_{i}\left(H_{j}\right)=\delta_{ij}$. Then with the
above definitions (\ref{eq:root operators for so9}), (\ref{eq:orthonormal Cartan})
and (\ref{eq:small root operators for so9}) one has
\[
\left[h,e_{\varepsilon_{1}L_{i}+\varepsilon_{2}L_{j}}\right]=\left(\varepsilon_{1}L_{i}+\varepsilon_{2}L_{j}\right)\left(h\right)e_{\varepsilon_{1}L_{i}+\varepsilon_{2}L_{j}},\ \left[h,e_{\pm L_{j}}\right]=\pm L_{j}\left(h\right)e_{\pm L_{j}}\  \forall\,h\in\mathfrak{h}_{B_{4}}
\]
Thus, (\ref{eq:root operators for so9}) and (\ref{eq:small root operators for so9})
provide a root space decomposition of $\mathfrak{k}\left(A_{8}\right)\left(\mathbb{C}\right)\cong B_{4}\left(\mathbb{C}\right)$
with respect to the Cartan subalgebra $\mathfrak{h}_{B_{4}}$ spanned by (\ref{eq:orthonormal Cartan}).
A corresponding Chevalley--Serre basis is given by 
\[
e_{i}\coloneqq e_{L_{i}-L_{i+1}},\ f_{i}\coloneqq e_{-L_{i}+L_{i+1}},\ h_{i}\coloneqq H_{i}-H_{i+1}\ \forall\,i=1,2,3,
\]
\[
e_{4}=e_{+L_{4}},\ f_{4}\coloneqq e_{-L_{4}},\ h_{4}\coloneqq 2H_{4}.
\]
\end{prop}

\begin{proof}
One verifies that (\ref{eq:root operators for so9}) and (\ref{eq:small root operators for so9})
are eigenvectors unde the adjoint action of $\mathfrak{h}_{B_{4}}$
with the correct eigenvalues by direct computation. This suffices
for a root space decomposition because one knows abstractly that this
exhausts $\mathfrak{k}\left(A_{8}\right)\left(\mathbb{C}\right)$
as its isomorphism type is known. Checking the Chevalley--Serre basis is
a matter of fixing suitable prefactors.
\end{proof}
Now set 
\begin{equation}
x_{\pm}\coloneqq i\left(x_{9}\mp i\left[x_{3},x_{9}\right]\right)\label{eq:def of X=00005Cpm}
\end{equation}
then one has 
\begin{equation}
\left[x_{+},x_{-}\right]=2H_{2},\ \left[H_{2},x_{\pm}\right]=\pm x_{\pm},\ \left[H_{i},x_{\pm}\right]=0\ \forall\,i\neq2,\label{eq:X-relations 1}
\end{equation}
and the Slodowy-type relations
\begin{align}
\left[x_{+},y\right]=0 & \quad \forall\,y\in\left\{ f_{1},e_{2}\right\} \cup\left\{ e_{3},e_{4},f_{3},f_{4}\right\}, \label{eq:X-relations 2}\\
\ad\left(x_{+}\right)^{3}\left(y\right)=0=\text{ad}\left(y\right)^{2}\left(x_{+}\right) & \quad \forall\,y\in\left\{ e_{1},f_{2}\right\} \cup\left\{ e_{3},e_{4},f_{3},f_{4}\right\}, \label{eq:X-relations 3}\\
\left[x_{-},y\right]=0 & \quad \forall\,y\in\left\{ e_{1},f_{2}\right\} \cup\left\{ e_{3},e_{4},f_{3},f_{4}\right\}, \label{eq:X-relations 4}\\
\ad\left(x_{-}\right)^{3}\left(y\right)=0=\text{ad}\left(y\right)^{2}\left(x_{-}\right) & \quad \forall\,y\in\left\{ f_{1},e_{2}\right\} \cup\left\{ e_{3},e_{4},f_{3},f_{4}\right\}, \label{eq:X-relations 5}
\end{align}
as well as additional relations that hold in $\mf{k}\left(\mf{e}_9\right)(\mathbb{C})$
\begin{equation}
\left[x_{+},e_{\varepsilon L_{1}-L_{2}}\right]=x_{\alpha_{2}+\alpha_{3}+\alpha_{9}}-i\varepsilon x_{\alpha_{1}+\alpha_{2}+\alpha_{3}+\alpha_{9}}=\left[x_{-},e_{\varepsilon L_{1}+L_{2}}\right]\label{eq:non-defining relation 1}
\end{equation}
\begin{equation}
\text{ad}\left(x_{+}\right)^{2}\left(e_{\varepsilon L_{1}-L_{2}}\right)=2e_{\varepsilon L_{1}+L_{2}},\ \text{ad}\left(x_{-}\right)^{2}\left(e_{\varepsilon L_{1}+L_{2}}\right)=2e_{\varepsilon L_{1}-L_{2}}\label{eq:non-defining relation 2}
\end{equation}
\begin{equation}
\left[x_{+},e_{+L_{2}+\varepsilon L_{3}}\right]=0=\left[x_{-},e_{-L_{2}+\varepsilon L_{3}}\right] \label{eq:non-defining relation 3}
\end{equation}
\begin{equation}
\left[x_{+},e_{-L_{2}+\varepsilon L_{3}}\right]=-\varepsilon x_{\alpha_{9}+\alpha_{3}+\alpha_{4}}-ix_{\alpha_{9}+\alpha_{3}+\alpha_{4}+\alpha_{5}}=-\left[x_{-},e_{L_{2}+\varepsilon L_{3}}\right] \label{eq:non-defining relation 4}
\end{equation}
\begin{equation}
\text{ad}\left(x_{+}\right)^{2}\left(e_{-L_{2}+\varepsilon L_{3}}\right)=-2e_{L_{2}+\varepsilon L_{3}},\ \text{ad}\left(x_{-}\right)^{2}\left(e_{L_{2}+\varepsilon L_{3}}\right)=-2e_{-L_{2}+\varepsilon L_{3}}. \label{eq:non-defining relation 5}
\end{equation}
Consider the Cartan matrix of $B_{4}$  and a GIM which we call $B_{4}^{\diamond}$ that extends it: 
\[
B_{4}=\begin{pmatrix}2 & -1 & 0 & 0\\
-1 & 2 & -1 & 0\\
0 & -1 & 2 & -1\\
0 & 0 & -2 & 2
\end{pmatrix},
\quad B_{4}^{\diamond}=\begin{pmatrix}2 & -2 & 2&0&0\\
-1 & 2 & -1 & 0 &0\\
1 & -1 & 2 & -1 & 0\\
0 & 0 & -1 & 2 & -1 \\
0 & 0 & 0 & -2 & 2
\end{pmatrix}
\]
The nontrivial Serre relations (for the $e_{i}$ only) of $B_4$ spell out
to be 
\[
\text{ad}\left(e_{i}\right)^{2}\left(e_{i+1}\right)=0=\text{ad}\left(e_{i+1}\right)^{2}\left(e_{i}\right)\ \forall\,i=1,2
\]
\[
\text{ad}\left(e_{3}\right)^{2}\left(e_{4}\right)=0=\text{ad}\left(e_{4}\right)^{3}\left(e_{3}\right).
\]

Denote the associated GIM-algebra to $B_4^\diamond$ over $\mathbb{C}$ by $\mathfrak{gim}\left(B_{4}^{\diamond}\right)\left(\mathbb{C}\right)$. The diagrams of $B_4$ and $B_4^{\diamond}$ are given in figure~\ref{fig:gim}.

\begin{figure}[t!]
\centering
\begin{picture}(260,50)
\thicklines
\multiput(10,10)(30,0){4}{\circle*{10}}
\put(10,10){\line(1,0){60}}
\put(70,8){\line(1,0){28}}
\put(70,12){\line(1,0){28}}
\put(90,10.2){\line(-1,-1){6}}
\put(90,9.8){\line(-1,1){6}}
\put(7,-5){$1$}
\put(37,-5){$2$}
\put(67,-5){$3$}
\put(97,-5){$4$}
\put(50,30){$B_4$}
\multiput(160,10)(30,0){4}{\circle*{10}}
\put(160,10){\line(1,0){60}}
\put(220,8){\line(1,0){28}}
\put(220,12){\line(1,0){28}}
\put(240,10.2){\line(-1,-1){6}}
\put(240,9.8){\line(-1,1){6}}
\put(157,-5){$1$}
\put(187,-5){$2$}
\put(217,-5){$3$}
\put(247,-5){$4$}
\put(175,39){\circle*{10}}
\put(160,14){\line(1,2){12}}
\put(163,12){\line(1,2){12}}
\put(169.2,27.5){\line(-3,-1){7}}
\put(168.8,27.5){\line(1,-2){2.9}}
\multiput(190,14)(-2.5,5){5}{\line(-1,2){1.5}}
\multiput(187,12)(-2.5,5){5}{\line(-1,2){1.5}}
\put(180.8,27.8){\line(3,-1){7}}
\put(181.2,27.8){\line(-1,-2){2.9}}
\put(160,38){$0$}
\put(205,30){$B_4^{\diamond}$}
\end{picture}
\caption{\label{fig:gim}\sl The diagrams associated to $B_{4}$ and $B_{4}^{\diamond}$ with labelling of nodes.}
\end{figure}
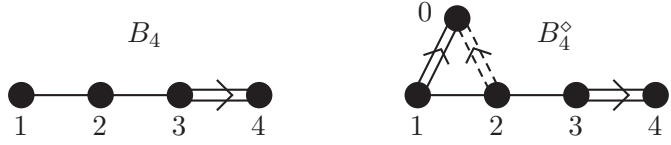

\begin{prop}
\label{prop:gim-quotient}Denote the Chevalley generators of $\mathfrak{gim}\left(B_{4}^{\diamond}\right)\left(\mathbb{C}\right)$
by $E_{0},\dots,E_{4}$, $F_{0},\dots,F_{4}$ and $H_{\gamma_{0}},\dots,H_{\gamma_{4}}$.
There exists a surjective homomorphism of Lie algebras $\phi:\mathfrak{gim}\left(B_{4}^{\diamond}\right)\left(\mathbb{C}\right)\rightarrow\mathfrak{k}\left(E_{9}\right)\left(\mathbb{C}\right)$
that is given on the Level of generators via 
\[
\phi\left(E_{0}\right)=x_{+},\ \phi\left(F_{0}\right)=x_{-},\ \phi\left(H_{\gamma_{0}}\right)=2H_{2}=2h_{2}+2h_{3}+h_{4}
\]
\[
\phi\left(E_{i}\right)=e_{i},\ \phi\left(F_{i}\right)=f_{i},\ \phi\left(H_{\gamma_{i}}\right)=h_{i}\ \forall\,i=1,2,3,4.
\]
\end{prop}

\begin{proof}
One verifies that the defining relations between the generators from
definition (\ref{def:GIM-algebra}) are satisfied. The $B_{4}$-relations
are unproblematic and towards the relations 
\[
\text{ad}\left(E_{0}\right)^{3}\left(E_{1}\right)=0=\text{ad}\left(E_{1}\right)^{2}\left(E_{0}\right),\ \text{ad}\left(E_{0}\right)^{3}\left(F_{2}\right)=0=\text{ad}\left(F_{2}\right)^{2}\left(E_{0}\right)
\]
 we refer to equations  (\ref{eq:X-relations 2})--(\ref{eq:X-relations 5}).
 One also has to
check that $H_{\gamma_{0}}\mapsto2H_{2}$ satisfies all necessary
identities which is the case. Thus, $\phi$ is a homomorphism of Lie-algebras.
Surjectivity follows from the fact that all Berman generators of $\mathfrak{k}\left(E_{9}\right)\left(\mathbb{C}\right)$
can be recovered from the image of the generators of $\mathfrak{gim}\left(B_{4}^{\diamond}\right)\left(\mathbb{C}\right)$.
For $x_{1},\dots,x_{8}$ this is a basis transformation within $B_{4}\left(\mathbb{C}\right)$
and for $x_{9}$ one notes that (\ref{eq:def of X=00005Cpm}) implies
\[
x_{+}+x_{-}=2ix_{9}.
\]
\end{proof}
GIM-algebras are graded with respect to their root system although their root system does not split
into positive and negative roots. Now consider an element $\left[E_{0},\left[E_{0},E_{1}\right]\right]-2E_{L_{1}+L_{2}}$
with $E_{L_{1}+L_{2}}\propto\left[E_{1},\left[E_{2},\left[E_{4},\left[E_{4},E_{3}\right]\right]\right]\right]$. This element is nonzero because the two summands lie in different root spaces and are nonzero themselves. But because of (\ref{eq:non-defining relation 2}) it is equal to $0$ in the image. This implies that $\mathfrak{k}\left(E_{9}\right)\left(\mathbb{C}\right)$
is a (nontrivial) quotient of $\mathfrak{gim}\left(B_{4}^{\diamond}\right)\left(\mathbb{C}\right)$. 
 
In conclusion, representations of $\mathfrak{gim}\left(B_{4}^{\diamond}\right)\left(\mathbb{C}\right)$
could potentially be useful to find representations of $\mathfrak{k}\left(E_{9}\right)\left(\mathbb{C}\right)$
if it is possible to check whether or not a given representation factors
through the projection of proposition \ref{prop:gim-quotient}. Conversely,
our results from sections \ref{sec:par} and \ref{sec:verm} provide representations of $\mathfrak{gim}\left(B_{4}^{\diamond}\right)\left(\mathbb{C}\right)$
both of finite and infinite dimension.


\section{\texorpdfstring{The Hilbert space completion $\widehat{\mf{k}}$ is not a Lie algebra}{The Hilbert space completion of k is not a Hilbert Lie algebra}}
\label{app:compl}

Because the restriction of the standard bilinear form is positive definite we can
complete $\mf{k}$ to a Hilbert space $\widehat{\mf{k}}$. Here we show by means 
of a simple explicit example that this Hilbert space completion is 
not compatible with the Lie algebra structure. Define
\begin{equation}
J_N(\omega) \coloneqq  \frac12 \omega^{IJ} \sum_{n=1}^N \frac1{n^{1/2+ \varepsilon}} \cX^{IJ}_n
\end{equation}
The positive definite bilinear form can be normalised such that (for $m,n\geq 0$)
\begin{equation}
\big\langle \cX_m^{IJ}  | \cX_n^{KL}  \big\rangle =  \delta_{mn} \delta^{IJ}_{KL}
\end{equation}
The induced norm is distinguished by its invariance, which implies that the right-hand side
is independent of $m$ and $n$. Consequently,
\begin{equation}
|\!| J_N |\!|^2 \,=\, C_0\sum_{n=1}^N \frac1{n^{1+2\varepsilon}}
\end{equation}
where $C_0 = C_0(\omega) $ is an irrelevant strictly positive constant. 
For $\varepsilon > 0$ this sum converges, 
and therefore the limit $J_\infty(\omega)  \equiv \lim_{N\rightarrow\infty} J_N(\omega)$ belongs 
to the Hilbert space $\widehat{\mf{k}}$. We next compute the commutator 
\begin{equation}\label{J1J2}
\Big[ J_N(\omega_1)\,,\, J_N(\omega_2) \Big] \,=\, 
\frac12 [\omega_1,\omega_2]^{IJ} \sum_{n=1}^N  f_n  \cX_n^{IJ}
\end{equation}
with 
\begin{equation}
f_n \,=\, \sum_{m=1}^{n-1} \frac1{m^{1/2+\varepsilon}} \frac1{(n-m)^{1/2+\varepsilon}}
+ \sum_{m=1}^N  \frac1{m^{1/2+\varepsilon}} \frac1{(n+m)^{1/2+\varepsilon}}
\end{equation}
Estimating the first sum on the right-hand side as ($n>1$)
\begin{equation}
 \sum_{m=1}^{n-1} \frac1{m^{1/2+\varepsilon}} \frac1{(n-m)^{1/2+\varepsilon}} \,>\,
 \frac{n-1}{n^{1 + 2\varepsilon}}
\end{equation}
it is easy to see that
\begin{equation}
f_n \,>\, \frac{C_1}{n^{2\varepsilon}}
\end{equation}
with another irrelevant strictly positive constant $C_1$. For $\varepsilon < \frac14$ the sum
$\sum_{n=1}^\infty |f_n|^2$ diverges, whence (\ref{J1J2}) does not converge in the
limit $N\rightarrow\infty$. In other words, although $J_\infty(\omega_1)$ and
$J_\infty(\omega_2)$ separately do belong to $\widehat{\mf{k}}$, their commutator 
does not exist as an element of $\widehat{\mf{k}}$. Hence $\widehat{\mf{k}}$ is not 
even a Lie algebra, and {\em a fortiori} also not a Hilbert Lie algebra, which would 
in addition require $|\!| [x,y] |\!| < C_2 |\!| x |\!| |\!| y |\!|$ for all $x,y \in \widehat{\mf{k}}$. 

We stress that the failure of $\widehat{\mf{k}}$ to be a Lie algebra depends on the norm used for the completion which in the analysis above was the standard invariant bilinear form. Other norms are possible and the corresponding completions of $\mf{k}$ can be Lie algebras, and
even Hilbert Lie algebras. However, in those cases the norm is not invariant. 

Although $\widehat{\mf{k}}$ is not a Lie algebra, an interesting open  question is 
whether one can still make sense of the commutator {\em as a distribution}, by considering
the commutator of two elements belonging to a dense subspace of $\widehat{\mf{k}}$.
This would fit with earlier observations in \cite{Kleinschmidt:2006dy}.


\section{\texorpdfstring{$\mf{so}(16)$ representations}{so(16) representations}}
\label{app:so16reps}

The translation of dimensions of $\mf{so}(16)$ representations to the labels of the highest weight in the conventions of the LiE software~\cite{LiE} are

\begin{multicols}{2}
\begin{tabular}{rcl}
${\bf 1}$ & $\leftrightarrow$ & [0,0,0,0,0,0,0,0]\\
${\bf 16}$ & $\leftrightarrow$ & [1,0,0,0,0,0,0,0]\\
${\bf 120}$ & $\leftrightarrow$ & [0,1,0,0,0,0,0,0]\\
${\bf 128}_s$ & $\leftrightarrow$ & [0,0,0,0,0,0,0,1]\\
${\bf 128}_c$ & $\leftrightarrow$ & [0,0,0,0,0,0,1,0]\\
${\bf 560}$ & $\leftrightarrow$ & [0,0,1,0,0,0,0,0]\\
${\bf 1344}$ & $\leftrightarrow$ & [1,1,0,0,0,0,0,0]\\
${\bf 1820}$ & $\leftrightarrow$ & [0,0,0,1,0,0,0,0]\\
${\bf 1920}_s$ & $\leftrightarrow$ & [1,0,0,0,0,0,0,1]\\
${\bf 4368}$ & $\leftrightarrow$ & [0,0,0,0,1,0,0,0]\\
${\bf 6435}_+$ & $\leftrightarrow$ & [0,0,0,0,0,0,0,2]\\
${\bf 7020}$ & $\leftrightarrow$ & [1,0,1,0,0,0,0,0]\\
${\bf 8008}$ & $\leftrightarrow$ & [0,0,0,0,0,1,0,0]\\
${\bf 11440}$ & $\leftrightarrow$ & [0,0,0,0,0,0,1,1]\\ 
${\bf 13312}_s$ & $\leftrightarrow$ & [0,1,0,0,0,0,0,1]\\
${\bf 13312}_c$ & $\leftrightarrow$ & [0,1,0,0,0,0,1,0]\\
${\bf 15360}_c$ & $\leftrightarrow$ & [2,0,0,0,0,0,1,0]\\
${\bf 24192} $ & $\leftrightarrow$ & [1,0,0,1,0,0,0,0]\\
\end{tabular}
\columnbreak

\begin{tabular}{rcl}
${\bf 56320}_s$ & $\leftrightarrow$ & [0,0,1,0,0,0,0,1]\\ 
${\bf 60060}$ & $\leftrightarrow$ & [1,0,0,0,1,0,0,0]\\
${\bf 91520}_+$ & $\leftrightarrow$ & [1,0,0,0,0,0,0,2]\\
${\bf 112320}$ & $\leftrightarrow$ & [1,0,0,0,0,1,0,0]\\
${\bf 141372}$ & $\leftrightarrow$ & [0,1,0,1,0,0,0,0]\\
${\bf 141440}_s$ & $\leftrightarrow$ & [1,1,0,0,0,0,0,1]\\
${\bf 161280}_s$ & $\leftrightarrow$ & [0,0,0,1,0,0,0,1]\\
${\bf 161280}_c$ & $\leftrightarrow$ & [0,0,0,1,0,0,1,0]\\
${\bf 162162}$ & $\leftrightarrow$ & [1,0,0,0,0,0,1,1]\\
${\bf 183040}_s$ & $\leftrightarrow$ & [0,0,0,0,0,0,0,3]\\
${\bf 326144}_s$ & $\leftrightarrow$ & [0,0,0,0,1,0,0,1]\\
${\bf 439296}_c$ & $\leftrightarrow$ & [0,0,0,0,0,0,1,2]\\
${\bf 465920}_c$ & $\leftrightarrow$  & [0,0,0,0,0,1,1,0]\\
${\bf 595595}_+$ & $\leftrightarrow$ & [0,1,0,0,0,0,0,2]\\
${\bf 670208}_s$ & $\leftrightarrow$ & [1,0,1,0,0,0,0,1]\\
${\bf 2036736}_s$ & $\leftrightarrow$ & [1,0,0,1,0,0,0,1]\\
${\bf 2489344}_s$ & $\leftrightarrow$ & [1,0,0,0,0,0,0,3] \\
${\bf 6223360}_s$ & $\leftrightarrow$ & [1,0,0,0,0,1,0,1]
\end{tabular}
\end{multicols}

\end{document}